\documentclass[11 pt]{amsart}

\usepackage{amsfonts}
\usepackage{amssymb}
\usepackage{amstext}
\usepackage{amsmath}
\usepackage[all]{xy}
\usepackage{graphicx}
\usepackage{color}
\usepackage{tikz}
\usetikzlibrary{knots, calc}
\usepackage{array}

\newtheorem{thm}{Theorem}
\newtheorem{cor}{Corollary}
\newtheorem{theorem}{Theorem}[section]
\newtheorem{lemma}[theorem]{Lemma}
\newtheorem{proposition}[theorem]{Proposition}

\newtheorem{corollary}[theorem]{Corollary}
\newtheorem{definition}[theorem]{Definition}

\def\S{\operatorname{Spin}^c}
\def\s{\text{spin}^c}
\def\relS{\operatorname{{Spin}^c}}
\def\intpt{\mathbb{T}_{\b\alpha}\cap\mathbb{T}_{\b\beta}}
\def\intptt{\mathbb{T}_{\b\alpha}\cap\mathbb{T}_{\b\gamma}}

\def\dis{\displaystyle}
\def\PD{\operatorname{PD}}

\def\b{\boldsymbol}
\def\abc{{\b\alpha\b\beta\b\gamma}}
\def\Z{\mathbb{Z}}
\def\Q{\mathbb{Q}}

\def\F{\mathbb{F}}

\def\ker{\operatorname{ker}}
\def\im{\operatorname{Im}}
\def\lcm{\operatorname{lcm}}

\def\Hom{\operatorname{Hom}}
\def\Tor{\operatorname{Tor}}

\def\can{\operatorname{can}}

\title{$\tau$-invariants for knots in rational homology spheres}
\author{Katherine Raoux}
\address{Katherine Raoux, Michigan State University, East Lansing, MI}
\email{raouxkat@msu.edu}

\begin{document}
\maketitle

\begin{abstract}
Ozsv\'ath and Szab\'o used the knot filtration on $\widehat{CF}(S^3)$ to define the $\tau$-invariant for knots in the 3-sphere. In this article, we generalize their construction and define a collection of $\tau$-invariants associated to a knot $K$ in a rational homology sphere $Y$. We then show that some of these invariants provide lower bounds for the genus of a surface with boundary $K$ properly embedded in a negative definite 4-manifold with boundary $Y$.
\end{abstract}

\section{Introduction} 
The $\tau$-invariant for knots in the 3-sphere defined by Ozsv\'ath and Szab\'o has proven to be a useful and robust invariant for studying knot concordance. A key property of this invariant is that $\tau(K)$ is a lower bound for the 4-ball genus of $K$ \cite{OS-4ball}. The goal of this article is to introduce a generalization, a collection of rational numbers $\{\tau_\mathfrak{s}(Y,K)\}_{\mathfrak{s}\in\S(Y)}$ associated to a knot $K$ in a rational homology sphere $Y$ and to prove that if $Y$ bounds an appropriate 4-manifold $W$, these invariants also provide lower bounds on the genus of a surface embedded in $W$ with boundary $K$.

Our main theorem concerns a knot $K$ in the boundary of a negative definite 4-manifold $W$. If $K$ is rationally null-homologous, there exists an integer $p$ such that $p[K]=0\in H_1(W;\Z)$, and we can consider a surface $\Sigma$ embedded in $W-\nu(K)$ with $[\partial\Sigma]=p[K]\in H_1(\nu(K);\Z)$. We call $\Sigma$ a \emph{p-slice surface} or a \emph{rational slice surface} for $K$. This gives a 4-dimensional notion of the more familiar concept of a rational Seifert surface for a knot in a 3-manifold.

A \emph{rational q-Seifert surface} for a rationally null-homologous knot $K$ in a 3-manifold $Y$ is a surface $F$ embedded in $Y-\nu(K)$ such that $[\partial F]=q[K]\in H_1(\nu(K);\Z)$. In particular, $q$ is a multiple of the order of $K$. Calegari and Gordon define the rational genus of a knot to be 
\begin{equation*}
||K||=\inf \frac{-\chi(F)}{2q}
\end{equation*}
where the infimum is taken over all $q$ and all $q$-Seifert surfaces without sphere components and obtain lower bounds on the rational genus for many knots in 3-manifolds \cite{Calegari-Gordon}. Using Heegaard Floer type invariants, Ni \cite{Ni-LinkFloer} and Ni and Wu \cite{Ni-Wu} have found further bounds on the rational genus. In particular, Ni showed that the knot Floer homology of a knot in a rational homology sphere detects the rational genus \cite{Ni-LinkFloer}. If $A_{\max}$ is the maximum Alexander grading of a nontrivial element of $\widehat{HFK}(Y,K)$, then 
\begin{equation}\label{3genus}
||K||=A_{\max}-\frac{1}{2}.
\end{equation}
By construction, our invariants are bounded above by $A_{\max}$. Therefore, we observe that $\tau_\mathfrak{s}(Y,K)\leq ||K||+\frac{1}{2}$. 

This raises the question, then, what happens when we consider $Y$ as the boundary of a 4-manifold? Any $q$-Seifert surface can be pushed into the 4-manifold to obtain a $q$-slicing surface. However, it is also possible that $K$ has a $p$-slicing surface for which $p<q$. This makes the relationship between the genera of rational Seifert surfaces and rational slicing surfaces for $K$ more subtle than the relationship between the 3-genus and 4-genus of knots in the 3-sphere.

Nevertheless, we establish that our invariants provide lower bounds for the genera of slice surfaces in the $p=1$ case:

\begin{thm}\label{maintheorem} Let $W$ be a negative definite 4-manifold and $K$ a knot in $\partial W=Y$. Let $F$ be a $q$-Seifert surface for $K$ and $\mathfrak{t}$ a sharp $\s$-structure on $W$. Then for any surface $\Sigma$ such that $\partial\Sigma=K$, 
\begin{equation*}
\frac{1}{q}\langle c_1(\mathfrak{t}),[q\Sigma\cup -F]\rangle+\frac{1}{q^2}[q\Sigma\cup -F]^2+2\tau_\mathfrak{s}(Y,K)\leq 2g(\Sigma)
\end{equation*}
where $\mathfrak{s}=\mathfrak{t}|_{Y}$.
\end{thm}

If $W$ is a rational homology 4-ball, the left-most terms vanish, so we have the following corollary: 
\begin{cor}\label{maincorollary}
Let $W$ be a rational ball with $\partial W=Y$, and $\Sigma$ a properly embedded surface in $W$ with $\partial\Sigma=K$. Then for each $\s$-structure $\mathfrak{s}$ on $Y$ that extends over $W$, $$|\tau_\mathfrak{s}(Y,K)|\leq g(\Sigma).$$
\end{cor}
We also show that our invariants satisfy properties similar to the original $\tau$-invariant. Therefore our invariants provide homomorphisms from the rational concordance group, as defined by Cha \cite{Cha}, to $\Q$. 

An alternative generalization of the $\tau$-invariant has recently been given by Ni and Vafaee in \cite{Ni-Vafaee}. However, our invariants satisfies the additional property that $\tau_{\mathfrak{s}}(-Y,K)=-\tau_{\mathfrak{s}}(Y,K)$.  For knots in $L$-spaces, our invariants coincide with theirs. For knots in lens spaces both coincide with the invariants defined by Celoria in \cite{Celoria}.

\subsection*{Organization} We begin in Section 2 with the construction of our invariants and conclude with a method for computing the Alexander grading from a Heegaard diagram. In Section 3 we turn to understanding how our invariants transform under change of orientation of $Y$ and $K$ as well as conjugation of $\mathfrak{s}$ and connected sums. Section 4 is devoted to describing the relationship between the knot filtration and certain maps on Floer homology determined by cobordisms between 3-manifolds. In Section 5 we give the proof of Theorem \ref{maintheorem}. Finally in Section 6 we compute some explicit examples.

\subsection*{Acknowledgements.} The author would like to thank Matthew Hedden, Adam Levine, and Thomas Mark for helpful conversations in the process of writing this paper. Thanks are also due to Jen Hom and Youlin Li for helpful comments on an earlier version of this paper and to Bryn Mawr College for its hospitality. Finally,
as the work here comprises my PhD thesis, I would especially like to thank my advisor Daniel Ruberman for his support and guidance and for suggesting this project in the first place.

\section{The definition of $\tau$}
Before giving the definition of the $\tau$-invariant, we first recall some facts about Heegaard Floer homology. For more detail, the reader may refer to \cite{OS-rationalsurgeries}. We adopt the orientation conventions of \cite{Hedden-Levine-surgery} throughout.

\subsection{Heegaard diagrams for $(Y,K)$} A \emph{doubly pointed Heegaard diagram} for an oriented knot $K$ in a 3-manifold $Y$ is a tuple $(\Sigma,\b\alpha,\b\beta,w,z)$ where $\Sigma$ is a genus $g$-surface, $\b\alpha$ and $\b\beta$ are each $g$-tuples of linearly independent simple closed curves on $\Sigma$, and $w$ and $z$ are basepoints on $\Sigma-\b\alpha-\b\beta$. The surface $\Sigma$ together with these sets of curves determine a Heegaard splitting of $Y$ as the union of two handlebodies $U_{\b\alpha}$ and $U_{\b\beta}$. As shown in \cite{OS-rationalsurgeries}, we can always construct such a doubly pointed Heegaard diagram from a particular choice of self-indexing Morse function $f$ on $Y$ so that the basepoints $w$ and $z$ determine the knot $K\subset Y$ with its orientation as the union of two gradient flow lines $K=\gamma_w-\gamma_z$ from the unique index 0 critical point to the unique index 3 critical point. Following convention, we assume that $U_{\b\alpha}=f^{-1}\left([0,\frac{3}{2}]\right)$ and $U_{\b\beta}=f^{-1}\left([\frac{3}{2},3]\right)$.

For calculations, it is convenient to represent $K$ as an oriented, possibly immersed, curve directly on the surface $\Sigma$. To do this, note that the knot $K=\gamma_w-\gamma_z$ can alternatively be decomposed into two arcs, one of which travels from $z$ to $w$ through $U_\alpha$ and the other from $w$ to $z$ through $U_\beta$. Now, we imagine letting the arc through $U_\alpha$ flow upward along the gradient, so that it lies on the surface as an arc $\nu_{\b\alpha}$ embedded in the complement of the $\b\alpha$-curves. Similarly, we let the arc through $U_\beta$ flow down until it lies as an arc $\nu_{\b\beta}$ embedded in the complement of the $\b\beta$-curves. The union $\nu_{\b\alpha}\cup\nu_{\b\beta}$ is now an immersed curve in the surface $\Sigma$ and the knot can be recovered from the diagram by remembering to which handlebody each arc belongs. 

Finally, for our arguments, we often find it useful to construct our Morse function in such a way that the final $\b\beta$-curve $\beta_g$ is a meridian for the knot that intersects only one of the $\b\alpha$-curves, $\alpha_g$ \cite{OS-knotinvariants}. In this case all of the generators of the Heegaard Floer chain complex, $\widehat{CF}(Y)$ are of the form $\mathbf{x}=(x_1,\ldots,x_{g-1},p)$ where $p$ is the unique intersection point between $\alpha_g$ and $\beta_g$.

\subsection{Relative $\s$-structures} In \cite{OS-rationalsurgeries}, Ozsv\'ath and Szab\'o show how to extend the correspondence between $\s$-structures and homology classes of non-vanishing vector fields introduced by Turaev \cite{Turaev-Spinc} to 3-manifolds $M$ with torus boundary components. Specifically, when $M=Y-\nu(K)$, a \emph{relative $\s$-structure} is a homology class of non-vanishing vector fields on $Y-\nu(K)$ that point outward on the boundary torus. There is an affine correspondence between relative $\s$-structures and classes in $H^2(Y-\nu(K),\partial\nu(K))\cong H^2(Y,K)$ which is analogous to the correspondence between $\S(Y)$ and $H^2(Y)$. We denote the set of relative $\s$-structures on $Y-\nu(K)$ by $\S(Y,K)$. 

In addition, there is a natural ``filling'' map 
\begin{equation*}
G_{Y,K}:\S(Y,K)\to\S(Y)
\end{equation*}
that can be described geometrically as follows. Let $\xi$ be a relative $\s$-structure on $Y$ and choose a representative vector field $v_\xi$ on $Y-\nu(K)$. Let $v_{\nu(K)}$ be a vector field on $\nu(K)\cong S^1\times D^2$ that agrees with $v_\xi$ on the boundary $\partial\nu(K)$ and smoothly extends over the interior in such a way that $K$ is a closed orbit and for each point $p$ in $S^1$ the vector field $v_{\nu(K)}$ is transverse to $p\times D_2$. Then $G_{Y,K}(\xi)$ is the homology class of the vector field on $Y$ that restricts to $v_{\xi}$ on $Y-\nu(K)$ and $v_{\nu(K)}$ on $\nu(K)$. 

The filling map is equivariant with respect to the action of $H^2$, in the sense that given an element $\alpha$ in $H^2(Y,K)$ 
\begin{equation*}
G_{Y,K}(\xi+\alpha)=G_{Y,K}(\xi)+i^*(\alpha).
\end{equation*}
where $i^*:H^2(Y,K)\to H^2(Y)$ is the map from the long exact sequence in cohomology.

\subsection{The Alexander grading and Alexander filtration}
Let $K$ be a knot in $Y$ and $(\Sigma,\b\alpha,\b\beta,w,z)$ be a corresponding doubly pointed Heegaard diagram. Then the set of relative $\s$-structures on $(Y,K)$ determines a filtration of the chain complex $\widehat{CF}(Y)$ via a map 
\begin{equation*}
\mathfrak{s}_{w,z}:\intpt\to \S(Y,K).
\end{equation*}
The construction of this map is described in \cite[Section 2.4]{OS-rationalsurgeries} and is similar to the construction of the map $\mathfrak{s}_w:\intpt\to \S(Y)$. 

When $K$ is a \emph{null-homologous} knot in $Y$, the filtration levels are labelled by integers via the Alexander grading \cite{OS-knotinvariants, Rasmussen-knots}. Later, Ni introduced a generalized Alexander grading for knots in rational homology spheres \cite{Ni-LinkFloer}. The definition we give uses the normalization and sign conventions of Hedden-Levine \cite{Hedden-Levine-surgery}. 

\begin{definition}
Let $K$ be a knot in a rational homology sphere $Y$ with corresponding doubly pointed Heegaard diagram $(\Sigma,\b\alpha,\b\beta,w,z)$. Fix a rational Seifert surface $F$ for $K$. For an intersection point $\mathbf{x}\in\intpt$ the \emph{Alexander grading} of $\mathbf{x}$ is given by
\begin{equation*}
A(\mathbf{x})=\frac{1}{2[\mu]\cdot[F]}(\langle c_1(\mathfrak{s}_{w,z}(\mathbf{x})),[F]\rangle+[\mu]\cdot[F]).
\end{equation*}
\end{definition}
For null-homologous knots, this definition coincides with the Alexander grading given by Rasmussen and Ozsv{\'a}th and Szab\'o. More generally, a pairing of this sort will exist for any rationally null-homologous knot $K$ in a 3-manifold. However, when $Y$ has $b_1>0$, the pairing will depend on the homology class of the rational Seifert surface $[F]$. In this article, we focus on knots in rational homology spheres, so the Alexander grading is independent of the choice of $F$.

The Alexander grading gives rise to a $\Q$-filtration of the Floer chain complex where $\mathcal{F}_q=\{\mathbf{x}\in\widehat{CF}(Y)| A(\mathbf{x})\leq q\}$ for each $q$ in $\Q$. Additionally, since $\widehat{CF}(Y)$ splits over $\s$-structures, the filtration splits as well. To define $\tau_\mathfrak{s}(Y,K)$, we are interested in the restriction of this filtration to a particular $\widehat{CF}(Y,\mathfrak{s})$ summand. Specifically, for each $\mathfrak{s}$ in $\S(Y)$, the $\Q$-valued Alexander filtration of $\widehat{CF}(Y,\mathfrak{s})$ can, in fact, be thought of as a $\Z$-filtration. 

To see this, fix a $\s$-structure $\mathfrak{s}$ in $\S(Y)$, and consider the set of lifts $G^{-1}_{Y,K}(\mathfrak{s})$. If $\xi$ and $\xi'$ are two lifts of $\mathfrak{s}$, then the difference $\xi'-\xi$ in $H^2(Y,K)$ is an element of $\ker(H^2(Y,K)\to H^2(Y))\cong \Z\langle \PD[\mu]\rangle$. Thus, $\xi'-\xi=m\PD[\mu]$ for some integer $m$. Equivalently, $\xi'=\xi+m\PD[\mu]$. Therefore, 
$$\langle c_1(\xi' ),[F]\rangle =\langle c_1(\xi),[F]\rangle+ 2m[\mu]\cdot[F].$$
Thus, if we consider all of elements of $G^{-1}_{Y,K}(\mathfrak{s})$ there is a unique choice of lift, which we denote by $\xi_0^\mathfrak{s}$ such that 
\begin{equation}\label{xi_0}
 -2[\mu]\cdot[F]< \langle c_1(\xi^\mathfrak{s}_0),[F]\rangle\leq 0.
\end{equation}

Now, for any $\mathbf{x}\in\intpt$ such that $\mathfrak{s}=\mathfrak{s}_w(\mathbf{x})=G_{Y,K}(\xi_0^\mathfrak{s})$ we know that $\mathfrak{s}_{w,z}(\mathbf{x})=\xi_0^\mathfrak{s}+m\PD[\mu]$ for some integer $m$. So we can write $A(\mathbf{x})=k_\mathfrak{s}+m$ where 
\begin{equation*}
k_\mathfrak{s}=\frac{1}{2[\mu]\cdot[F]}(\langle c_1(\xi_0^\mathfrak{s}),[F]\rangle +[\mu]\cdot[F]). 
\end{equation*}
From Equation \ref{xi_0} we have that $k_\mathfrak{s}$ is rational and lies in the interval $(-\frac{1}{2},\frac{1}{2}]$. The number $k_\mathfrak{s}$ also depends on the orientation of $K$ since the orientation of $K$ determines the orientation of $\mu$ and $F$.

Now we can see that filtration coming from the Alexander grading is really a $\Z$-filtration of $\widehat{CF}(Y,\mathfrak{s})$ where 
\begin{equation*}
\mathcal{F}_{\mathfrak{s}, m}=\{\mathbf{x}\in\widehat{CF}(Y,\mathfrak{s}) | A(\mathbf{x})-k_\mathfrak{s}=m\}
\end{equation*}
and if $q$ is any rational number, the restriction of the filtration level $\mathcal{F}_{q}$ to a particular $\widehat{CF}(Y,\mathfrak{s})$ summand of $\widehat{CF}(Y)$ is $\mathcal{F}_{\mathfrak{s},m}$ where $m$ is the largest integer such that $k_\mathfrak{s}+m\leq q$.

We first define $\tau_\mathfrak{s}(Y,K)$ in the simpler case that $Y$ is an $L$-space. Recall that $Y$ is an $L$-space if $\widehat{HF}(Y,\mathfrak{s})\cong\F_2$ for each $\s$-structure $\mathfrak{s}$ in $\S(Y)$. Then 
$$\tau_{\mathfrak{s}}(Y,K)= \min \{m |\; I_m:H_*(\mathcal{F}_{\mathfrak{s},m})\to \widehat{HF}(Y,\mathfrak{s}) \text{ is non-trivial }\}.$$ 

Note that when $Y$ is an $L$-space, $I_m$ is nontrivial if and only if $I_m$ is surjective. For a general 3-manifold $Y$, this may not be the case since $\widehat{HF}(Y,\mathfrak{s})$ may be a vector space of dimension greater than one. Thus, to define $\tau$ in general, we will need to use more of the structure of $CF^\infty(Y,\mathfrak{s})$.

\subsection{Some facts about $CF^\infty(Y)$} Let $Y$ be a 3-manifold with a pointed Heegaard diagram $(\Sigma,\b\alpha,\b\beta,w)$.  Recall from \cite{OS-threemanifolds} that $CF^\infty(Y,\mathfrak{s})$ is freely generated by elements $[\mathbf{x},i]$ where $\mathbf{x}\in\intpt$, $i\in\Z$, and the subcomplex $CF^-(Y)$ is obtained by restricting to the pairs $[\mathbf{x},i]$ where $i<0$. These fit into a short exact sequence 
\begin{equation*}
0\to CF^-(Y,\mathfrak{s})\xrightarrow{\iota}CF^\infty(Y,\mathfrak{s})\xrightarrow{\pi} CF^+(Y,\mathfrak{s})\to0
\end{equation*}
where $CF^+(Y)$ is the quotient complex. 
 
In addition, we have a chain map $U:CF^\infty(Y,\mathfrak{s})\to CF^\infty(Y,\mathfrak{s})$ defined by $[\mathbf{x},i]\mapsto[\mathbf{x},i-1]$ which induces an isomorphism on homology. The $U$ action on $CF^\infty(Y,\mathfrak{s})$ allows us to view $CF^\infty(Y,\mathfrak{s})$ as a finitely generated $\F_2[U,U^{-1}]$-module. 

The restriction of $U$ to either $CF^-(Y,\mathfrak{s})$ or $CF^+(Y,\mathfrak{s})$ does \emph{not} induce an isomorphism on homology. In fact, the complex $\widehat{CF}(Y,\mathfrak{s})$ is the kernel of $U$ applied to $CF^+(Y,\mathfrak{s})$. Thus we have another short exact sequence:
\begin{equation*}
0\to\widehat{CF}(Y,\mathfrak{s}) \xrightarrow{\rho} CF^+(Y,\mathfrak{s}) \xrightarrow{U}  CF^+(Y,\mathfrak{s})\to 0.
\end{equation*}

Alternatively we can view $\widehat{CF}(Y,\mathfrak{s})$ as a cokernel. For any $n\in\Z$, let $CF^{<n}(Y,\mathfrak{s})$ be the subcomplex of $CF^\infty(Y,\mathfrak{s})$ generated by $[\mathbf{x},i]$ where $i<n$. Then, $\widehat{CF}(Y,\mathfrak{s})$ also fits into the short exact sequence:
\begin{equation*}
0\to CF^{<1}(Y,\mathfrak{s})\xrightarrow{U} CF^{<1}(Y,\mathfrak{s}) \xrightarrow{\psi} \widehat{CF}(Y,\mathfrak{s})\to 0.
\end{equation*}

In addition, $U$ maps $CF^{<1}(Y,\mathfrak{s})$ isomorphically onto $CF^-(Y,\mathfrak{s})$. In particular, we have an isomorphism of short exact sequences:
\begin{equation*}
\xymatrix{
0\ar[r]& CF^{<1}(Y,\mathfrak{s})\ar[r]^{\hat\iota}\ar[d]^U &CF^\infty(Y,\mathfrak{s})\ar[r]^{\hat\pi}\ar[d]^U &CF^{\geq 1}(Y,\mathfrak{s})\ar[r]\ar[d]^U &0\\
0\ar[r]& CF^-(Y,\mathfrak{s})\ar[r]^{\iota}& CF^\infty(Y,\mathfrak{s})\ar[r]^{\pi}&CF^+(Y,\mathfrak{s})\ar[r]&0.
}
\end{equation*}

Let $HF^{<n}(Y,\mathfrak{s})$ denote the homology of the complex $CF^{<n}(Y,\mathfrak{s})$. Then, putting the above information together, we have that 
\begin{equation*}
\xymatrix{
HF^{<1}(Y,\mathfrak{s})\ar[r]^{\psi_*}\ar[dr]_{\hat{\iota}_*}& \widehat{HF}(Y,\mathfrak{s})\ar[r]^{\rho_*}& HF^+(Y,\mathfrak{s})\\
& HF^\infty(Y,\mathfrak{s})\ar[ur]_{\pi_*}&\\
}
\end{equation*}
commutes.

\subsection{The definition of $\tau_{\mathfrak{s}}(Y,K)$.}
Fix a $\s$-structure $\mathfrak{s}\in\S(Y)$, then for each $m\in\Z$ we have a short exact sequence:
\begin{equation*}
0\to \mathcal{F}_{\mathfrak{s},m}\xrightarrow{i_m}\widehat{CF}(Y,\mathfrak{s})\xrightarrow{p_m} \mathcal{Q}_{\mathfrak{s},m}\to 0
\end{equation*}
where $\mathcal{Q}_{\mathfrak{s},m}$ is the quotient. We denote the maps induced on homology by $i_m$ and $p_m$ as $I_m$ and $P_m$ respectively.

\begin{definition}
Let $K\subset Y$ be a knot in a rational homology sphere. Then
\begin{equation*}
\tau_{\mathfrak{s}}(Y,K):=\min\{k_\mathfrak{s}+m\:|\: \im(\rho_*\circ I_m)\cap\im(\pi_*)\neq 0\}
\end{equation*}
where $\rho_*:\widehat{HF}(Y,\mathfrak{s})\to HF^+(Y,\mathfrak{s})$ and $\pi_*:HF^\infty(Y,\mathfrak{s})\to HF^+(Y,\mathfrak{s})$ are the maps induced on homology by $\rho$ and $\pi$. 
\end{definition}

Putting together several exact triangles, we form the diagram: 
\begin{equation*}
\begin{tikzpicture}

\node at (-3,1.5){$H_*(\mathcal{F}_{\mathfrak{s},m})$};
\node at (1,1.5){$H_*(\mathcal{Q}_{\mathfrak{s},m})$};

\node at (-5,0){$HF^{<1}(Y, \mathfrak{s})$};
\node at (-1,0){$\widehat{HF}(Y, \mathfrak{s})$};
\node at (5,0){$HF^+(Y, \mathfrak{s})$};

\node at (0,-1.5){$HF^{<1}(Y, \mathfrak{s})\cong HF^-(Y, \mathfrak{s})$};

\node at (1,-3){$HF^\infty(Y, \mathfrak{s})$};

\draw[->] (-2.75,1.25)--(-1.75,.3); 
\node at (-2,1){$I_m$};
\draw[<-] (.75,1.25)--(-.35,.25); 
\node at (0,1){$P_m$};

\draw[->] (-3.9,0)--(-2,0); 
\node at (-2.75,.25){$\psi_*$};
\draw[->] (0,0)--(4,0); 
\node at (2,.25){$\rho_*$};

\draw[->] (-1,-.35)--(-1,-1.25); 
\node at (-.75,-.75){$\partial'$};
\draw[->] (1,-1.75)--(1,-2.75); 
\node at (1.25,-2.25){$\iota_*$};

\draw[->] (4.5,-.25)--(2,-1.25); 
\node at (3.5,-1){$\partial$};
\draw[<-] (-4,-.25)--(-2,-1.25); 
\node at (-3,-1.1){$U$};

\draw[->] (-5,-.25)..controls +(down: 14mm) and +(left: 14mm)..(0,-3); 
\node at (-3,-2.5){$\hat\iota_*$};
\draw[<-] (5,-.25)..controls +(down: 10mm) and +(right: 10mm)..(2,-3); 
\node at (4,-2.5){$\pi_*$};

\end{tikzpicture}
\end{equation*}
which commutes.

\begin{lemma}\label{TFAE} Fix a $\s$-structure $\mathfrak{s}$ in $\S(Y)$ and an integer $m$. Then the following are equivalent:
\begin{enumerate}
\item There exists an element $\beta$ in $H_*(\mathcal{F}_{\mathfrak{s},m})$ such that $\rho_*\circ I_m(\beta)\neq 0$ and $\rho_*\circ I_m(\beta)$ is in $\im(\pi_*)$.
\item There exists a non $U$-torsion element $\alpha$ in $HF^{<1}(Y,\mathfrak{s})$ and an element $\beta$ in $H_*(\mathcal{F}_{\mathfrak{s},m})$ such that $I_m(\beta)=\psi_*(\alpha)\neq 0$.
\item There exists an element $\alpha$ in $HF^{<1}(Y,\mathfrak{s})$ such that $\pi_*\circ\hat\iota(\alpha)\neq0$ and $P_m(\psi_*(\alpha))=0$. 
\end{enumerate}
\end{lemma}

\begin{proof}
Assuming (1), there exists an element $\beta$ in $H_*(\mathcal{F}_{\mathfrak{s},m})$ such that $\rho_*\circ I_m(\beta)\neq 0$ and $\rho_*\circ I_m(\beta)$ is in $\im(\pi_*)$. By commutativity of the diagram and exactness, $\partial\circ\rho_*\circ I_m(\beta)=\partial'\circ I_m(\beta)=0$. Thus, there exists some element $\alpha$ in $HF^{<1}(Y,\mathfrak{s})$ so that $I_m(\beta)=\psi_*(\alpha)$. In addition, $\rho_*\circ I_m(\beta)=\rho_*\circ\psi_*(\alpha)=\pi_*\circ\hat\iota_*(\alpha)$ so we must have that $\alpha$ is non-torsion.

Now assume (2). Since $\alpha$ is non-torsion, and $I_m(\beta)=\psi_*(\alpha)\neq 0$, we must have that $\pi_*\circ\hat\iota(\alpha)\neq 0$. Thus, $\rho_*\circ\psi_*(\alpha)=\rho_*\circ I_m(\beta)\neq 0\in\im(\pi_*)$. Thus, $\psi_*(\alpha)$ is in $\ker(P_m)$.

Finally, (3) implies (1) since $\psi_*(\alpha)\neq 0$ and $\psi_*(\alpha)$ is in $\ker(P_m)=\im(I_m)$. 
\end{proof}

\noindent The lemma suggests an alternative formulation of $\tau_{\mathfrak{s}}(Y,K)$ in terms of $HF^{<1}(Y,\mathfrak{s})$. 
\begin{definition} Let $K\subset Y$ be a knot in a rational homology sphere. Then
\begin{equation*}
\tau^-_{\mathfrak{s}}(Y,K):=\max_{\substack{\alpha\in HF^{<1}(Y,\mathfrak{s})\\ \pi_*\circ \hat\iota_*(\alpha)\neq 0}}\{k_\mathfrak{s}+m\; |\;  P_{m}\circ\psi_*(\alpha)\neq 0\}
\end{equation*}
\end{definition}
The two definitions are related in the following way:
\begin{proposition} \label{tau-}
Let $Y$ be a rational homology 3-sphere, and $K\subset Y$. Then for each $\mathfrak{s}\in\S(Y)$,
\begin{equation*}
\tau_\mathfrak{s}^-(Y,K)=\tau_\mathfrak{s}(Y,K)-1.
\end{equation*}
\end{proposition}
		
\begin{proof}
Suppose $k_\mathfrak{s}+m=\tau_{\mathfrak{s}}(Y,K)$. Then by Lemma \ref{TFAE}, $k_\mathfrak{s}+m'$ where $m'=m-1$ must be the maximum for which every $ \alpha\in HF^{<1}(Y,\mathfrak{s})$ such that $\pi_*\circ \hat\iota_*(\alpha)\neq 0$ satisfies $P_{m'}\circ\psi_*(\alpha)\neq 0$. 
\end{proof}

\subsection{Framings for rationally null-homologous knots}\label{framings}
	For any oriented knot $K$ in any 3-manifold there is a well defined \emph{meridian} $\mu_K$, namely the homology class of the curve that the generates the kernel of the map 
\begin{equation*}H_1(\partial\nu(K))\to H_1(\nu(K)).\end{equation*}
Geometrically, the meridian is the curve in $\partial\nu(K)$ that bounds a disk in $\nu(K)$ and is oriented so that its intersection number with $K$ is $+1$. A \emph{longitude} for $K$ is any choice of curve $\lambda$ such that $(\mu_K,\lambda)$ forms a basis for $H_1(\partial\nu(K))\cong \Z\oplus\Z$. When $K$ is null-homologous, there is a canonical choice of longitude called the Seifert-framing or the 0-framing given by the curve $\lambda_0=F\cap \partial\nu(K)$, where $F$ is a Seifert surface for $K$. 

If $K$ is only rationally null-homologous, we use a rational Seifert surface to determine a \emph{canonical longitude} for $K$, following Mark and Tosun \cite{Mark-Tosun}. First notice that the map
\begin{equation*}
i_*:H_1(\partial\nu(K))\to H_1(Y-\nu(K))
\end{equation*}
induced by inclusion has kernel isomorphic to $\Z$.

In particular, if $F$ is a rational Seifert surface for $K$, then $F\cap\partial\nu(K)=\partial F$ is a set of closed curves on $\partial \nu(K)$. While $[\partial F]$ itself may not be a primitive element in $H_1(\partial \nu(K))$, there exists an integer $c$ and a primitive element $\gamma$ so that $[\partial F]=c\gamma$. The integer $c$ is called the \emph{complexity} of $F$ and is equal to the number of boundary components of $F$. If $F$ is a $q$-Seifert surface, and $\mu$ is the meridian of $K$ we have $[\partial F]\cdot [\mu]=c[\gamma]\cdot[\mu]=q$. Therefore, if the complexity $c$ is equal to $q$, then $\gamma$ will be an honest longitude for $K$. Otherwise, for any choice of longitude $\lambda$ we have the following decomposition:
\begin{equation*}
[\partial F]=c\gamma=c(d\lambda+r\mu)
\end{equation*}
where $cd=q$ and $r$ is an integer. Note that the number $\frac{r}{d}$ in $\Q\slash\Z$ is the rational self linking number of $K$. 

This decomposition depends on the choice of longitude in the sense that any other choice of longitude for $K$ can be written as $\lambda_m=\lambda+m\mu$ for some integer $m$. Changing basis from $(\lambda,\mu)$ to $(\lambda_m,\mu)$ we have: $$[\partial F]=c(d(\lambda_m-m\mu)+r\mu)=c(d\lambda_m+(r-md)\mu)=c(d\lambda_m+r'\mu).$$ As we can see $r'=r-md$ differs from $r$ by an integer multiple of $d$. Thus, there is a unique choice of longitude, which we call the \emph{canonical longitude} $\lambda_{\can}$ such that 
\begin{equation*}
[\partial F]=c(d\lambda_{\can}+r\mu) \quad\quad \text{and} \quad\quad 0\leq r<d.
\end{equation*} 
In other words, $\lambda_{\can}$ is the choice of longitude for which $\frac{r}{d}$ is the unique representative of the rational self-linking number of $K$ in $[0,1)$.

\subsection{Relative periodic domains and the $c_1$-evaluation formula} The following shows how to compute the Alexander grading of a generator directly from a doubly pointed Heegaard diagram for $(Y,K)$.

Recall that for knots in the 3-sphere we can compute the Alexander grading using periodic domains on the Heegaard diagram for $0$-surgery along $K$ \cite{OS-knotinvariants}. Here we describe an analogous construction for knots in rational homology spheres. 

Let $(Y,K)$ be a rationally null-homologous knot with doubly pointed Heegaard diagram $(\Sigma,\b\alpha,\b\beta,w,z)$ constructed so that $\beta_g=\mu$ and $\beta_g$ intersects the $\b\alpha$-curves in a single point on $\alpha_g$. Let $\lambda$ be a curve on $\Sigma$ which is a longitude for $K$. The following definition is given by Hedden and Plamenevskaya \cite{Hedden-Plam}.

\begin{definition}
Let $(Y,K)$ be a rationally null-homologous knot in a 3-manifold with Heegaard diagram and longitude $\lambda$ as described above. Let  $D_1,D_2,\ldots D_r$ be the closures of the components of $\Sigma-(\b\alpha\cup\b\beta\cup\lambda)$. A \emph{relative periodic domain} is a relative 2-chain $\mathcal{P}=\sum_i a_iD_i$ with
\begin{equation*}
\partial\mathcal{P}=q\lambda+\sum_{i=1}^gn_{\alpha_i}\alpha_i+\sum_{i=1}^gn_{\beta_i}\beta_i
\end{equation*}
where the coefficients $a_i$ are the local multiplicities of $\mathcal{P}$. 
\end{definition}

A relative periodic domain gives rise to a map $\Phi:F\to\Sigma$ where $F$ is an oriented surface with boundary such that $\partial F$ maps into $\b\alpha\cup\b\beta\cup\b\gamma$. 

\begin{lemma}\label{periodicdom}
Given a rational q-Seifert surface, $F$ for $K$, satisfying $\partial F=q\lambda_{\can}+cr\mu$ we can always find a relative periodic domain $\mathcal{P}_F$ for $F$ which satisfies 
\begin{equation*}
\partial \mathcal{P}_{F}=q\lambda_{\can}+cr\mu+q\alpha_g+\sum_{i=1}^{g-1} n_{\alpha_i}\alpha_i +\sum_{i=1}^{g-1} n_{\beta_i}\beta_i
\end{equation*}
where $\lambda_{\can}$ is the canonical longitude for $K$.
\end{lemma}
\begin{proof}
Fix a Heegaard diagram for $(Y,K)$ where $\beta_{g}=\mu$ and $\beta_g$ and $\alpha_g$ intersect in a single point $p$. For any choice of longitude, there is a small neighborhood of $\mu$ where $\lambda\cap\alpha_g=\varnothing$ as shown in Figure \ref{WR}.
\begin{figure}
\begin{center}
\includegraphics[scale=1]{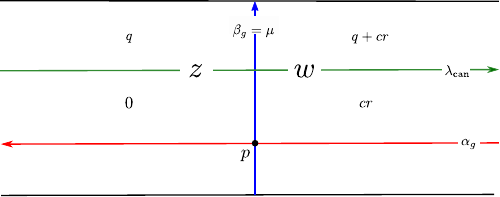}
\caption{A small neighborhood of $\mu$.}
\label{WR}
\end{center}
\end{figure}

Let $\mathcal{P}_F$ be any periodic domain for $F$. Then by definition,
\begin{align*}
\partial\mathcal{P}_F&=q\lambda+\sum_{i=1}^gn_{\alpha_i}\alpha_i+\sum_{i=1}^gn_{\beta_i}\beta_i\\
	&=q\lambda+n_{\alpha_g}\alpha_g+n_{\beta_g}\mu+\sum_{i=1}^{g-1}n_{\alpha_i}\alpha_i+\sum_{i=1}^{g-1}n_{\beta_i}\beta_i.
\end{align*}
Since $\lambda=\lambda_{\can}+n\mu$ for some integer $n$ we can change coordinates:\begin{align*}
\partial\mathcal{P}_F&=q(\lambda_{\can}+n\mu)+n_{\alpha_g}\alpha_g+n_{\beta_g}\mu+\sum_{i=1}^{g-1}n_{\alpha_i}\alpha_i+\sum_{i=1}^{g-1}n_{\beta_i}\beta_i.
\end{align*}

Now, since $\mathcal{P}_F$ represents $F$, we should be able to cap off $\mathcal{P}_F$ with sums of $\b\alpha$ and $\b\beta$ curves to obtain a representative of $F$ with $[\partial F]=q\lambda_{\can}+cr\mu$. Therefore, we must have $n_{\beta_g}=cr-nq$ and from inspecting the diagram in Figure \ref{WR} we have that $n_{\alpha_g}=q$.
\end{proof}

In addition, any relative periodic domain gives rise to a relative homology class in $H_2(Y-\nu(K),\partial(Y-\nu(K))$. Recall that when $K$ is null-homologous, there is a $c_1$-evaluation formula for periodic domains \cite{OS-threemanifolds}. Hedden and Levine provide a similar formula for relative periodic domains. First, recall the definition of the Euler measure:
\begin{definition} [\cite{OS-properties}]
Let $\mathcal{P}=\sum_i a_iD_i$ be a (relative) periodic domain. The \emph{Euler measure of $\mathcal{P}$} is given by,
\begin{equation*}
\hat\chi(\mathcal{P})=\sum_i n_i(\chi(D_i)-\frac{1}{4}\# (\text{corners in }D_i)).
\end{equation*}
\end{definition}
\noindent Then, we have a $c_1$-evaluation formula for relative periodic domains: 
\begin{proposition}[\cite{Hedden-Levine-surgery}]\label{Hedden-Levine}
Let $K$ be a rationally null-homologous knot in $Y$ with rational $q$-Seifert surface $F$. Let $(\Sigma,\b\alpha,\b\beta,w,z)$ be a doubly pointed Heegaard diagram for $(Y,K)$. For any relative periodic domain $\mathcal{P}$ representing $F$ and any $\mathbf{x}$ in $\intpt$,
\begin{equation*}
\langle c_1(\mathfrak{s}_{w,z}(\mathbf{x})),[F]\rangle+q=\hat\chi(\mathcal{P})+2n_{\mathbf{x}}(\mathcal{P})-\bar{n}_{w,z}(\mathcal{P}).
\end{equation*}
\end{proposition}
\noindent Here $\bar{n}_{w,z}(\mathcal{P})$ is the sum of the averages of the multiplicities of $\mathcal{P}$ on either side of the $w$ and $z$ base points.

\section{Knot Floer homology for rationally null-homologous knots}

To understand how $\tau_{\mathfrak{s}}(Y,K)$ transforms under certain operations, such as reversing the orientation of $K$ and $Y$, we must first understand how the knot Floer chain complex transforms under these operations. We begin by recalling the construction of $CFK^\infty(Y,K,\xi)$.

\subsection{Knot Floer chain complexes} Following \cite[Section 3]{OS-rationalsurgeries}, we fix a doubly pointed Heegaard diagram $(\Sigma,\b\alpha,\b\beta,w,z)$ for $(Y,K)$ and a relative $\s$-structure $\xi$ on $(Y,K)$. Then the complex $CFK^\infty(Y,K,\xi)$ is generated by triples $[\mathbf{x},i,j]$ where $\mathbf{x}$ is in $\intpt$ and $i,j$ are integers such that 
\begin{equation*}
\mathfrak{s}_{w,z}(\mathbf{x})+(i-j)\PD[\mu]=\xi
\end{equation*} and is equipped with the usual differential,
\begin{equation*}
\partial^\infty[\mathbf{x},i,j]=\sum_{\mathbf{y}\in\intpt}\sum_{\big\{\substack{\phi\in\pi_2(\mathbf{x},\mathbf{y})\\ \mu(\phi)=1}\big\}} \widehat{\mathcal{M}}(\phi)[\mathbf{y},i-n_w(\phi),j-n_z(\phi)].
\end{equation*}

\begin{proposition}[{\cite[Proposition 3.2]{OS-rationalsurgeries}}]\label{SES}
For any $\xi\in\relS(Y,K)$ there is a natural identification of certain subcomplexes of $CFK^\infty(Y,K,\xi)=\mathcal{C}_\xi$ with the Floer chain complexes associated to $Y$. Specifically, 
\begin{equation*}
0\to \mathcal{C}_\xi\{i<0\}\to \mathcal{C}_\xi\to \mathcal{C}_\xi\{i\geq0\}\to 0
\end{equation*}
and 
\begin{equation*}
0\to \mathcal{C}_\xi\{i=0\}\to \mathcal{C}_\xi\{i\geq0\}\xrightarrow{U} \mathcal{C}_\xi\{i\geq0\}\to 0
\end{equation*}
can be naturally identified with
\begin{equation*}
0\to CF^-(Y,\mathfrak{s})\to CF^\infty(Y,\mathfrak{s})\to CF^+(Y,\mathfrak{s})\to 0
\end{equation*}
and 
\begin{equation*}
0\to \widehat{CF}(Y,\mathfrak{s})\to CF^+(Y,\mathfrak{s})\xrightarrow{U} {CF^+}(Y,\mathfrak{s})\to 0
\end{equation*}
where $G_{Y,K}(\xi)=\mathfrak{s}$. 
In addition,
\begin{equation*}
0\to \mathcal{C}_\xi\{j<0\}\to \mathcal{C}_\xi\to \mathcal{C}_\xi\{j\geq0\}\to 0
\end{equation*}
and 
\begin{equation*}
0\to \mathcal{C}_\xi\{j=0\}\to \mathcal{C}_\xi\{j\geq0\}\to \mathcal{C}_\xi\{j\geq0\}\to 0
\end{equation*}
can be naturally identified with
\begin{equation*}
0\to CF^-(Y,\mathfrak{s}')\to CF^\infty(Y,\mathfrak{s}')\to CF^+(Y,\mathfrak{s}')\to 0
\end{equation*}
and
\begin{equation*}
0\to \widehat{CF}(Y,\mathfrak{s}')\to CF^+(Y,\mathfrak{s}')\xrightarrow{U} {CF^+}(Y,\mathfrak{s}')\to 0
\end{equation*}
where $G_{Y,K^r}(\xi)=\mathfrak{s}'$. 
\end{proposition}

\emph{A priori} the complex $CFK^\infty(Y,K,\xi)$ depends on the relative $\s$-structure $\xi$. However, Ozsv\'ath and Szab\'o show that if $\xi_1$ and $\xi_2$ in $\relS(Y,K)$ map to the same $\mathfrak{s}\in\S(Y)$ under $G_{Y,K}$, then $CFK^\infty(Y,K,\xi_1)$ and $CFK^\infty(Y,K,\xi_2)$ differ by a shift in their $j$-filtration. 

Since we have picked a distinguished element of $\xi_0^\mathfrak{s}\in G^{-1}_{Y,K}(\mathfrak{s})$ for each $\mathfrak{s}$, we define the complex: 
\begin{equation*}
CFK^\infty(Y,K,\mathfrak{s}):=CFK^\infty(Y,K,\xi_0^\mathfrak{s}).
\end{equation*}

\subsection{Symmetries of $CFK^\infty(Y,K,\mathfrak{s})$} We are interested in how the complexes $CFK^\infty(Y,K,\mathfrak{s})$ change under orientation reversal of $Y$ and $K$ as well as under conjugation of $\mathfrak{s}$. Following convention, we use $CFK^*_\infty(Y,K,\mathfrak{s})$ to denote the dual complex, $\Hom_{\F_2}(CFK_*^\infty(Y,K,\mathfrak{s}),\F_2)$. 
\begin{lemma}\label{dualcomplex}
Let $Y$ be a rational homology 3-sphere and $K\subset Y$ a knot. Let $-Y$ denote $Y$ with its reverse orientation. Then for any $\mathfrak{s}\in\S(Y)$, 
\begin{equation*}
CFK^\infty_*(-Y,K,\mathfrak{s})\cong CFK^*_\infty (Y,K,\mathfrak{s}).
\end{equation*}
\end{lemma}
	
\begin{proof}
Let $(\Sigma,\b\alpha,\b\beta,w,z)$ be a Heegaard diagram for $(Y,K)$. Then reversing the orientation of $\Sigma$ gives a Heegaard diagram $(-\Sigma,\b\alpha,\b\beta,w,z)$ for $(-Y,K)$. For $\mathbf{x}$ in $\intpt$ the map $\mathfrak{s}_w(\mathbf{x})$ is independent of the orientation of $\Sigma$, and thus the orientation of $Y$. Following \cite[Proposition 2.5]{OS-properties}, suppose $\mathbf{x}, \mathbf{y}$ are in $\intpt$ and $\phi$ is a disk in $\pi_2(\mathbf{x},\mathbf{y})$. There is a natural identification 
\begin{equation*}
\widehat{\mathcal{M}}_{J_s}(\phi)=\widehat{\mathcal{M}}_{-J_s}(\phi')
\end{equation*}
where $\phi'\in\pi_2(\mathbf{y},\mathbf{x})$ is the class satisfying $n_z(\phi')=n_z(\phi)$ obtained by precomposing with complex conjugation. 

Moreover, since $n_z(\phi)=n_z(\phi')$ we also have $n_w(\phi)=n_w(\phi')$, and, 
\begin{align*}
\mathfrak{s}_{w,z}(\mathbf{y})-\mathfrak{s}_{w,z}(\mathbf{x})&=(n_{z}(\phi')-n_w(\phi'))\PD[\mu]\\
	&=(n_z(\phi)-n_w(\phi))\PD[\mu].
\end{align*}	
Thus, the duality map 
\begin{equation*}
\mathcal{D}: CFK^\infty_*(Y,K,\mathfrak{s})\to CFK_\infty^*(-Y,K,\mathfrak{s})
\end{equation*}
taking elements $[\mathbf{x},i,j]\mapsto [\mathbf{x},-i,-j]$ is an isomorphism. 
\end{proof}
	
To understand the dependence on the orientation of the knot $K\subset Y$ recall that there is an involution on relative $\s$-structures 
\begin{equation*}
\tilde{J}:\relS(Y,K)\to\relS(Y,K)
\end{equation*}
and the related involution, 
\begin{equation*}
J:\S(Y)\to\S(Y)
\end{equation*}
on $\s$-structures. If $v$ is a vector field, representing $\mathfrak{s}\in\S(Y)$, then $-v$ represents $J\mathfrak{s}$. This involution behaves appropriately with respect to the filling maps in the following sense,
\begin{lemma}
\begin{equation*}
G_{Y,K^r}(\tilde{J}\xi)=JG_{Y,K}(\xi).
\end{equation*}
\end{lemma}
\begin{proof}
As in \cite[Section 3.7]{OS-multivariable} we represent a relative $\s$-structure $\xi$ in $\relS(Y,K)$ by a nowhere zero vector field $v$ that has $K$ as a closed orbit. Thus, $-v$ is a vector field representing $\tilde{J}\xi$ that has $K^r$ as a closed orbit. 
\end{proof}
To show how the complex $CFK^\infty(Y,K,\mathfrak{s})$ transforms under reversal of orientation of $K$, we first investigate how $CFK^\infty(Y,K,\xi)$ transforms under this orientation reversal where $\xi$ is a relative $\s$-structure. 
\begin{lemma} [{\cite[Lemma 3.12]{OS-multivariable}}]\label{reverse}
Let $(\Sigma,\b\alpha,\b\beta,w,z)$ be a Heegaard diagram for $(Y,K)$, then $(-\Sigma,\b\beta, \b\alpha,w,z)$ is a diagram for $(Y,K^r)$. Let 
\begin{equation*}
\mathfrak{s}_{w,z}:\intpt\to \S(Y,K)
\end{equation*}
be the map determined by $(\Sigma,\b\alpha,\b\beta,w,z)$, and 
\begin{equation*}
\mathfrak{s}'_{w,z}:\intpt \to\S(Y,K)
\end{equation*} 
be the map determined by $(-\Sigma,\b\beta, \b\alpha,w,z)$. Then, 
\begin{equation*}
\mathfrak{s}_{w,z}(\mathbf{x})=\tilde J\mathfrak{s}'_{w,z}(\mathbf{x}).
\end{equation*}
\end{lemma}
\begin{proof}
If $f$ is a Morse function compatible with $(\Sigma,\b\alpha,\b\beta,w,z)$ then $-f$ is compatible with $(-\Sigma,\b\beta,\b\alpha,w,z)$. Now, if $v$ is a vector field representing $\mathfrak{s}_{w,z}(\mathbf{x})$, then $-v$ will represent $\mathfrak{s}'_{w,z}(\mathbf{x})$. \end{proof}

Applying the lemma, we obtain the following, which is the analogue of \cite[Proposition 3.9]{OS-knotinvariants}:

\begin{proposition}
For each $\xi\in\relS(Y,K)$, we have 
\begin{equation*}
CFK^\infty(Y,K,\xi)\cong CFK^\infty(Y,K^r,\tilde{J}\xi). 
\end{equation*}
\end{proposition}
\begin{proof}
Suppose $[\mathbf{x},i,j]$ is a generator of $CFK^\infty(Y,K,\xi)$, then 
\begin{equation*}
\mathfrak{s}_{w,z}(\mathbf{x})+(i-j)\PD[\mu_K]=\xi
\end{equation*}
where $\mu_K$ is the meridian of $K$. On the other hand by Lemma \ref{reverse} we know $\mathfrak{s}_{w,z}(\mathbf{x})=\tilde{J}\mathfrak{s}'_{w,z}(\mathbf{x})$,  therefore 
\begin{equation*}
\tilde{J}\mathfrak{s}'_{w,z}(\mathbf{x})+(i-j)\PD[\mu_{K}]=\xi
\end{equation*}
which is equivalent to 
\begin{equation*}
\mathfrak{s}'_{w,z}(\mathbf{x})+(i-j)\PD[\mu_{K^r}]=\tilde{J}\xi
\end{equation*}
where $\mu_{K^r}$ is the meridian of $K^r$. Moreover, if $\phi\in\pi_2(\mathbf{x},\mathbf{y})$ is a holomorphic disk contributing to the differential in $CFK^\infty(Y,K,\xi)$ it is still holomorphic in the diagram $(-\Sigma, \b\beta, \b\alpha ,w,z)$ and contributes to the differential in $CFK^\infty(Y,K^r,\tilde{J}\xi)$.
\end{proof}

Thus we may conclude that in fact,
\begin{proposition} \label{conjugation}
\begin{equation*}
CFK^\infty(Y,K,\mathfrak{s})\cong CFK^\infty(Y,K^r,J\mathfrak{s}).
\end{equation*}
\end{proposition}
\begin{proof} 
Using the previous proposition, we know that 
\begin{equation*}
CFK^\infty(Y,K,\mathfrak{s})=CFK^\infty(Y,K,\xi_0^\mathfrak{s})\cong CFK^\infty(Y,K^r,\tilde{J}\xi_0^\mathfrak{s}).
\end{equation*} Thus, we only need to show that $\xi_0^{J\mathfrak{s}}=\tilde{J}\xi_0^{\mathfrak{s}}$. 

Note that $-2[\mu]\cdot[F]<\langle c_1(\xi_0^\mathfrak{s}),[F]\rangle\leq 0$, where $F$ is a rational Seifert surface for $K$. Also, note that $c_1(\tilde{J}\xi_0^\mathfrak{s})=-c_1(\xi_0^\mathfrak{s})$. Thus, we have 
\begin{equation*}
-2[\mu_{K^r}]\cdot [-F]=-2[\mu_{K}]\cdot [F]<\langle c_1(\tilde{J}\xi_0^\mathfrak{s}),-[F]\rangle=\langle c_1(\xi_0^\mathfrak{s}),[F]\rangle \leq 0
\end{equation*}
where $-F$ is a rational Seifert surface associated to $K^r$.
\end{proof}

\subsection{Connected sums}
We are also interested in how $\tau_\mathfrak{s}(Y,K)$ behaves under connected sum. For a pair $(Y_1,K_1)$ and $(Y_2,K_2)$ of knots in 3-manifolds, we form their connected sum, $(Y_1,K_1)\#(Y_2,K_2):=(Y_1\#Y_2,K_1\#K_2)$. Given corresponding Heegaard diagrams $(\Sigma_1,\b{\alpha}_1,\b\beta_1,w_1,z_1)$ and $(\Sigma_2,\b{\alpha}_2,\b\beta_2,w_2,z_2)$ for $(Y_1,K_1)$ and $(Y_2,K_2)$ a Heegaard diagram for the connected sum is given by the tuple $(\Sigma_1\#\Sigma_2,\b{\alpha}_1\cup\b\alpha_2,\b\beta_1\cup\b\beta_2,w_1,z_2)$ where the connected sum of $\Sigma_1$ and $\Sigma_2$ is performed by identifying neighborhoods of $w_2$ and $z_1$. 
 
We can also ``glue'' $\s$-structures over the connected sum to obtain a map:
\begin{equation*}
\S(Y_1,K_1)\times\S(Y_2,K_2)\to \S(Y_1\#Y_2,K_1\#K_2)
\end{equation*} 
which sends $(\xi_1,\xi_2)\mapsto \xi_1\#\xi_2$ and is equivariant with respect to the action of 
\begin{equation*}
H^2(Y_1,K_1)\oplus H^2(Y_2,K_2)\to H^2(Y_1\#Y_2,K_1\#K_2).
\end{equation*}
In addition, given an intersection point $\mathbf{x}_1\otimes\mathbf{x}_2\in\mathbb{T}_{\b\alpha_1\cup\b\alpha_2}\cap\mathbb{T}_{\b\beta_1\cup\b\beta_2}$, 
\begin{equation*}
\mathfrak{s}_{w_1,z_2}(\mathbf{x}_1\otimes\mathbf{x}_2)=\mathfrak{s}_{w_1,z_1}(\mathbf{x}_1)\#\mathfrak{s}_{w_2,z_2}(\mathbf{x}_2).
\end{equation*}
The following theorem comes from \cite[Theorem 5.1]{OS-rationalsurgeries}, however the statement there seems to have a typographical error. The correct statement is contained in \cite[Theorem 7.1]{OS-knotinvariants}. We also include it here for clarity.
\begin{theorem}
Fix $\xi_3\in\relS(Y_1\# Y_2,K_1\# K_2)$. There is a filtered chain homotopy equivalence 
\begin{multline*}
\bigoplus_{\xi_1\#\xi_2=\xi_3} CFK^\infty(Y_1,K_1,\xi_1)\otimes_{\F_2[U,U^{-1}]}CFK^\infty(Y_2,K_2, \xi_2)\\ 
\overset{\cong}{\longrightarrow} CFK^\infty(Y_1\#Y_2,K_1\#K_2, \xi_3).
\end{multline*}
\end{theorem}
\begin{proof}The map sending $[\mathbf{x}_1,i_1,j_1]\otimes[\mathbf{x}_2,i_2,j_2]$ to
\begin{small}
\begin{equation*}
\sum_{\mathbf{y}\in\mathbb{T}_{\b\alpha_1\cup\b\alpha_2}\cap\mathbb{T}_{\b\beta_1\cup\b\beta_2}}\sum_{\{\psi\in\pi_2(\mathbf{x}_1\otimes\theta_2,\theta_1\otimes\mathbf{x}_2,\mathbf{y})\}} \#\hat{\mathcal{M}}(\psi)\cdot[\mathbf{y},i_1+i_2-n_{w_1}(\psi),j_1+j_2-n_{z_2}(\psi)]
\end{equation*}
\end{small}defines the isomorphism and the proof is the same as the proof given in \cite[Theorem 7.1]{OS-knotinvariants}.
\end{proof}
For knots in rational homology spheres, we will improve this to a statement about filtered complexes. To do so, it remains to show how the connected sum operation effects the filtration defined by the Alexander grading. 

If $K_1$ has order $q_1$ and $K_2$ has order $q_2$, their connected sum has order $\lcm(q_1,q_2)$ in $Y_1\#Y_2$. Following Calegari and Gordon \cite{Calegari-Gordon}, given connected rational Seifert surfaces $F_1$ and $F_2$, we can construct a rational Seifert surface $F$ for $K_1\#K_2$ by taking $q_1$ copies of $F_2$ and $q_2$ copies of $F_1$ and taking their boundary connected sum along $q_1q_2$ arcs. The resulting surface $F$ has Euler characteristic
\begin{equation*}
\chi(F)=q_2\chi(F_1)+q_1\chi(F_2)-q_1q_2.
\end{equation*}

\begin{lemma}\label{additive}
Let $K_1\subset Y_1$ and $K_2\subset Y_2$ be knots in rational homology spheres. Then  if $\mathbf{x}_1\otimes\mathbf{x}_2$ is an intersection point in the Heegaard diagram for $(Y_1\#Y_2,K_1\#K_2)$,
\begin{equation*}
A(\mathbf{x}_1\otimes\mathbf{x}_2)=A(\mathbf{x}_1)+A(\mathbf{x}_2).
\end{equation*}
\end{lemma}
\begin{proof} Let $q_1$ and $q_2$ denote the order of $K_1$ and $K_2$, and fix a connected rational Seifert surface for each of $K_1$ and $K_2$. Let $\mathcal{P}_1$ and $\mathcal{P}_2$ be periodic domains representing $F_1$ and $F_2$ respectively. Following the construction above, we can construct a periodic domain $\mathcal{P}$ for a rational Seifert surface $F$ for $K_1\#K_2$ by taking $q_2$ copies of $\mathcal{P}_1$ and $q_1$ copies of $\mathcal{P}_2$ and connecting them along $q_1q_2$ arcs. Now we can compute $A(\mathbf{x}_1\otimes \mathbf{x}_2)$ using Proposition \ref{Hedden-Levine} and $\mathcal{P}$. 

\begin{figure}
\begin{center}
\includegraphics[scale=.7]{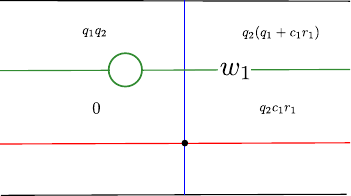}\hspace{1cm}
\includegraphics[scale=.7]{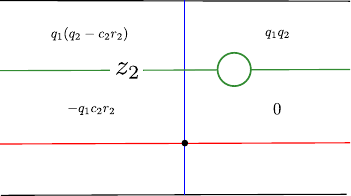}
\caption{Neighborhoods of the meridian of $K_1$ and $K_2$ respectively.}
\label{connectsum}
\end{center}
\end{figure}

From the diagram in Figure \ref{connectsum} we have:
\begin{itemize}
\item $\hat\chi(\mathcal{P})=q_2\hat\chi(\mathcal{P}_1)+q_1\hat\chi(\mathcal{P}_2)-q_1q_2$.
\item $n_{\mathbf{x}_1\otimes \mathbf{x}_2}(\mathcal{P})=q_2n_{\mathbf{x}_1}(\mathcal{P}_1)+q_1n_{\mathbf{x}_2}(\mathcal{P}_2)$.
\item $n_{w_1,z_2}(\mathcal{P})=q_1q_2+q_2c_1r_1-q_1c_2r_2$.
\end{itemize}
Putting this together, we have 
\begin{align*}
2q_1q_2A(\mathbf{x}_1\otimes\mathbf{x}_2)&=\hat\chi(\mathcal{P})+2n_{\mathbf{x}_1\otimes\mathbf{x}_2}(\mathcal{P})-n_{w_1,z_2}(\mathcal{P})\\
	&= q_2(\hat\chi(\mathcal{P}_1)+2n_{\mathbf{x}_1}(\mathcal{P}_1)-q_1-c_1r_1)\\
	&\quad\quad \quad\quad+q_1(\hat\chi(\mathcal{P}_2)+2n_{\mathbf{x}_2}(\mathcal{P}_2)-q_2+c_2r_2)\\
											&=q_2(2q_1A(\mathbf{x}_1))+q_1(2q_2A(\mathbf{x}_2)).
\end{align*}
This proves the claim. 
\end{proof}

Letting $\mathcal{F}_{\mathfrak{s}}(\mathbf{x})$ denote the highest Alexander filtration level inhabited by the generator $\mathbf{x}$.  Lemma \ref{additive} implies the following:

\begin{theorem} We have a filtered chain homotopy equivalence
\begin{small}
\begin{equation*}
CFK^\infty(Y_1\# Y_2, K_1\# K_2,\mathfrak{s})\cong\bigoplus_{\mathfrak{s}_1\#\mathfrak{s}_2=\mathfrak{s}}CFK^\infty(Y_1,K_1,\mathfrak{s}_1)\otimes_{\F_2[U,U^{-1}]} CFK^\infty(Y_2,K_2,\mathfrak{s}_2) 
\end{equation*}
\end{small}
where $\mathcal{F}_{\mathfrak{s}_1\#\mathfrak{s}_2}(\mathbf{x}_1\otimes\mathbf{x}_2)=\mathcal{F}_{\mathfrak{s}_1}(\mathbf{x}_1)+\mathcal{F}_{\mathfrak{s}_2}(\mathbf{x}_2)$
\end{theorem}

\subsection{Properties of $\tau$}
Our invariants satisfy the following properties: 				
\begin{proposition} \label{properties}
Let $K\subset Y$ be a knot in a rational homology sphere. Then for any $\mathfrak{s}\in\S(Y)$, we have:
\begin{enumerate}		
\item $\tau_\mathfrak{s}(-Y,K)=-\tau_\mathfrak{s}(Y,K)$.
\item $\tau_{\mathfrak{s}}(Y,K^r)=\tau_{J\mathfrak{s}}(Y,K)$.
\item If $K_1$ and $K_2$ are knots in rational homology spheres $Y_1$ and $Y_2$. Then for $\s$-structures $\mathfrak{s}_1\in\S(Y_1)$ and $\mathfrak{s}_2\in\S(Y_2)$, we have 
\begin{equation*}
\tau_{\mathfrak{s}_1\#\mathfrak{s}_2}(Y_1\# Y_2, K_1\# K_2)=\tau_{\mathfrak{s}_1}(Y_1,K_1)+\tau_{\mathfrak{s}_2}(Y_2,K_2).
\end{equation*}
\end{enumerate}
\end{proposition}

\begin{proof}[Proof of Proposition \ref{properties}]

First we note that property (2) follows directly from Proposition \ref{conjugation}.

Now consider property (1). The proof is similar in spirit to the proof of \cite[Lemma 3.3]{OS-4ball}. For simplicity we use $\mathcal{F}_{r}(Y,K,\mathfrak{s})$ to denote the Alexander filtration at level $r=k_\mathfrak{s}+m$ of $\widehat{CF}(Y,\mathfrak{s})$ coming from the knot $K$ and $\mathcal{Q}_{r}(Y,K,\mathfrak{s})$ for the corresponding quotient complex.

As in the case where $K$ is null-homologous \cite{OS-4ball}, reversing the orientation of $Y$ changes the sign of the Alexander grading of each generator in $\widehat{CF}(Y,\mathfrak{s})$. In the short exact sequence,
\begin{equation*}
0\to\mathcal{F}_{r}(Y,K,\mathfrak{s})\to\widehat{CF}(Y,\mathfrak{s})\to \mathcal{Q}_{r}(Y,K,\mathfrak{s})\to 0
\end{equation*}
we can naturally identify $\mathcal{Q}_{r}(Y,K,\mathfrak{s})$ with $\mathcal{F}^*_{-r-1}(-Y,K,\mathfrak{s})$. In fact, we have an isomorphism of short exact sequences:
\begin{equation*}
\xymatrix{
0\ar[r]& \mathcal{F}_{r}(Y,K,\mathfrak{s})\ar[r]^-{i_{r}}\ar[d]& \widehat{CF}_*(Y,\mathfrak{s})\ar[r]^-{p_{r}}\ar[d]&\mathcal{Q}_{r}(Y,K,\mathfrak{s})\ar[r]\ar[d]&0\\
0\ar[r]& \mathcal{Q}^*_{-r-1}(-Y,K,\mathfrak{s})\ar[r]^-{p^*_{-r-1}}& \widehat{CF}^*(-Y,\mathfrak{s})\ar[r]^-{i^*_{-r-1}}&\mathcal{F}^*_{-r-1}(-Y,K,\mathfrak{s})\ar[r]&0
}
\end{equation*}
by Lemma \ref{dualcomplex}.

Since $\tau_{\mathfrak{s}}(Y,K)$ is defined in terms of the map $\rho_*$ to $HF^+(Y,\mathfrak{s})$ we must also understand how this map transforms when we reverse the orientation of $Y$. Recall from \cite{OS-properties} that $HF_*^+(Y,\mathfrak{s})\cong HF_{<1}^*(-Y,\mathfrak{s})$. Thus, we have the following diagram:
\begin{equation*}
\xymatrix{H_*(\mathcal{F}_{r}(Y,K,\mathfrak{s}))\ar[r]^-{I_{r}}\ar[d]^{\cong}& \widehat{HF}_*(Y,\mathfrak{s})\ar[r]^-{\rho_*}\ar[d]^{\cong} 
&HF^+_*(Y,\mathfrak{s})\ar[d]^{\cong}\\
H^*(\mathcal{Q}^*_{-r-1}(-Y,K,\mathfrak{s}))\ar[r]^-{P^*_{-r-1}}&\widehat{HF}^*(-Y,\mathfrak{s})\ar[r]^-{\psi^*}& HF^*_{<1}(-Y\mathfrak{s})
}\label{diagram}
\end{equation*}
which commutes. 

Note that $r\geq \tau_\mathfrak{s}(Y,K)$ if and only if $\im(\rho_*\circ I_r)$ is non-torsion. By the above diagrams, we see that $\im(\rho_*\circ I_r)=\im(\psi^*\circ P_{-r-1})$. Therefore, $\im(\rho_*\circ I_r)$ is non-torsion if and only if $\im(\psi^*\circ P_{-r-1})$ is non-torsion. By the Universal Coefficients Theorem and the fact that for a rational homology sphere, $\Hom_{\F_2}(HF^{<1}_*(-Y,\mathfrak{s}),\F_2)=\F_2[U]$ we see that $\im(\psi^*\circ P_{-r-1})$ is non-torsion if and only if there exists some non-torsion $\alpha\in HF^{<1}(-Y,\mathfrak{s})$ so that $P_{-r-1}\circ\psi_*(\alpha)\neq 0$, that is, if and only if $-r-1\leq \tau^{-}(-Y,K)$. Thus, $\tau^-(-Y,K)=-\tau(Y,K)-1$, so by Proposition \ref{tau-}, $\tau(-Y,K)=-\tau(Y,K)$. 
\\

The final argument for property (3) was suggested to me by Jen Hom and is due to Tye Lidman. We begin by proving the following inequality:
$$\tau_{\mathfrak{s}_1\#\mathfrak{s}_2}(Y_1\# Y_2,K_1\#K_2)\leq\tau_{\mathfrak{s}_1}(Y_1,K_1)+\tau_{\mathfrak{s}_2}(Y_2,K_2).$$
Let $r_1=\tau_{\mathfrak{s}_1}(Y_1,K_1)$ and $r_2=\tau_{\mathfrak{s}_2}(Y_2,K_2)$. Then for each $i=1,2$, there exists $\alpha_i\in HF^{<1}(Y_i,\mathfrak{s}_i)$ non-torsion and $\beta_i\in H_*(\mathcal{F}_{\mathfrak{s}_i,r_i}(Y_i,K_i))$ so that $\psi_*(\alpha_i)=I_{r_i}(\beta_i)\neq 0$. Since $\alpha_1$ and $\alpha_2$ are both non-torsion, $\alpha_1\otimes\alpha_2$ is also non-torsion. Moreover, this diagram
\begin{equation*}
\xymatrix{
HF^{<1}(Y_1,\mathfrak{s}_1)\otimes_{\F_2[U]} HF^{<1}(Y_2,\mathfrak{s}_2)\ar[r]\ar[d]^{\psi_*\otimes\psi_*}& HF^{<1}(Y_1\# Y_2,\mathfrak{s}_1\# \mathfrak{s}_2)\ar[d]^{\psi_*}\\
\widehat{HF}(Y_1,\mathfrak{s}_1)\otimes_{\F_2}\widehat{HF}(Y_2,\mathfrak{s}_2)\ar[r]^{\cong}&\widehat{HF}(Y_1\# Y_2,\mathfrak{s}_1\#\mathfrak{s}_2) 
}
\end{equation*}
commutes. (The bottom isomorphism is due to the fact that we tensored over $\F_2$ so the relevant $\Tor$-groups vanish.) Thus, 
\begin{equation*}
I_{r_1}(\beta_1)\otimes I_{r_2}(\beta_2)=\psi_*(\alpha_1)\otimes\psi_*(\alpha_2)=\psi_*(\alpha_1\otimes\alpha_2).
\end{equation*}
On the other hand,  
\begin{equation*}
H_*(\mathcal{F}_{\mathfrak{s}_1,r_1}(Y_1,K_1))\otimes_{\F_2}H_*(\mathcal{F}_{\mathfrak{s}_2,r_2}(Y_2,K_2))\cong H_*(\mathcal{F}_{\mathfrak{s}_1,r_1}(Y_1,K_1)\otimes_{\F_2}\mathcal{F}_{\mathfrak{s}_2,r_2}(Y_2,K_2))
\end{equation*}
therefore
\begin{equation*}
I_{r_1}(\beta_1)\otimes I_{r_2}(\beta_2)=I_{r_1+r_2}(\beta_1\otimes\beta_2).
\end{equation*}
and $I_{r_1+r_2}(\beta_1\otimes\beta_2)\neq 0$ since both $I_{r_i}(\beta_i)\neq 0$.

Thus, $\alpha_1\otimes \alpha_2$ is non-torsion and $\psi_*(\alpha_1\otimes\alpha_2)=I_{r_1+r_2}(\beta_1\otimes\beta_2)\neq 0$. This implies that $r_1+r_2\geq \tau_{\mathfrak{s}_1\#\mathfrak{s}_2}(Y_1\# Y_2,K_1\# K_2)$.

To obtain the other inequality, note that the above also implies that,
$$\tau_{\mathfrak{s}_1\#\mathfrak{s}_2}(-Y_1\# -Y_2,K_1\#K_2)\leq\tau_{\mathfrak{s}_1}(-Y_1,K_1)+\tau_{\mathfrak{s}_2}(-Y_2,K_2).$$
Applying property (1), we obtain the desired result. 

\end{proof}

It is worth pointing out that in Proposition \ref{properties} (1) and (3) are the same as for knots in the 3-sphere. Furthermore, if $K$ is a knot in the 3-sphere (2) implies that reversing the orientation of the knot does not change $\tau(K)$ since there is only one $\s$-structure on $S^3$. In particular, for $K\subset S^3$, properties (1) and (2) together imply that $\tau(-K)=-\tau(K)$ where $-K$ is the reverse mirror of $K$, or equivalently, $-K$ is the inverse of $K$ in the knot concordance group.

\section{Large Surgery}

Ozsv\'ath and Szab\'o proved that the Heegaard Floer homology of surgery along a null-homologous knot can be computed from the chain complex $CFK^\infty(Y,K)$ using their ``large surgery" formula \cite{OS-knotinvariants}. In \cite{OS-rationalsurgeries}, they give a similar construction for rationally null-homologous knots and any choice of longitude. We apply their theorem using $\lambda_{\can}$ as our choice of longitude. Specifically, this choice of longitude allows us to enumerate 2-handle cobordism maps on Floer homology using the Alexander grading of $K$ and, in turn, extract 4-dimensional information from our $\tau$-invariants.

Let $K$ be a knot in a rational homology sphere $Y$ and let $Y_{-n}(K)$ denote the 3-manifold obtained by  performing $(-n)$-surgery to $Y$ along $K$ with respect to the canonical longitude. Let $X_{-n}(K)$ be the 4-manifold with $\partial X_{-n}(K)=-Y\sqcup Y_{-n}(K)$ obtained by attaching a 4-dimensional 2-handle to $K\times \{1\}\subset Y\times I$ with $(-n)$-framing with respect to the canonical longitude.

For each $\s$-structure $\mathfrak{t}$ in $\S(X_{-n}(K))$, this 4-manifold induces a map on Floer homology:
\begin{equation*}
F_{X_{-n}(K), \mathfrak{t}}:\widehat{HF}(Y,\mathfrak{t}|_Y)\to\widehat{HF}(Y_{-n}(K),\mathfrak{t}|_{Y_{-n}(K)}).
\end{equation*}
The large surgery theorem will tell us how to compute the above map from the $CFK^\infty(Y, K,\mathfrak{s})$ complexes. We can enumerate these maps by enumerating $\s$-structures on $X_{-n}(K)$. To do so, first calculate some more information about the topology of $X_{-n}(K)$. 

Let $C$ be the core of the added 2-handle in $X_{-n}(K)$. Then $[C]$ represents the generator of $H_2(X_{-n}(K),Y)\cong \Z$. Now fix a rational $q$-Seifert surface $F$ for $K$ in $Y$. Attaching $q$ parallel copies of $-C$ to $F$ along $K$ we obtain a 2-complex $F_{qC}=F\cup -qC$ which represents a homology class in $H_2(X_{-n}(K))$. Under the map 
\begin{equation*}
\iota:H_2(X_{-n}(K))\to H_2(X_{-n}(K),Y)
\end{equation*}
coming from the long exact sequence of the pair $(X_{-n}(K), Y)$ we have that $\iota([F_{ qC}])=-q[C]$. Thus, for a cohomology class $\alpha$ in $H^2(X_{-n}(K);\Z)$ we can define a $\Q$-valued pairing: 
\begin{equation*}
\langle \alpha,-[C]\rangle=\frac{1}{q}\langle \alpha,[F_{qC}]\rangle.
\end{equation*} 
Similarly, we also have,
\begin{equation*}
[C]\cdot[C]=\frac{1}{q^2}[F_{qC}]\cdot [F_{qC}].
\end{equation*}
Since $[\partial F]=c(d\lambda_{\can}+r\mu)$, we can compute $[F_{qC}]^2$ directly as the intersection number of $[\partial F]$ and $q[-\partial C]$ in $\partial\nu(K)$. 
\begin{equation}\label{selfintersection}
[\partial F]\cdot q[-\partial C]= q(-cr-nq).
\end{equation}

\subsection{Large surgery along rationally null-homologous knots} 
Choose a Heegaard diagram $(\Sigma,\b\alpha,\b\beta, w, z)$ for $(Y,K)$ so that $\b\beta=\b\beta_0\cup\mu$. Here $\b\beta_0=(\beta_1,\ldots,\beta_{g-1})$ and $\alpha_g$ and $\beta_g=\mu$ intersect in a single point $p$. Now consider an annular neighborhood of $\mu$ in the Heegaard diagram $(\Sigma,\b\alpha,\b\beta,w,z)$. Let $\lambda_{-n}=\lambda_{\can}-n\mu$ and let $\b\gamma$ denote the tuple $\b\beta_0\cup\lambda_{-n}$. Then $(\Sigma,\b\alpha,\b\gamma, w,z)$ is a Heegaard diagram for $Y_{-n}(K)$. In particular, $\lambda_{-n}$ intersects $\alpha_g$ in exactly $n$ points in a small neighborhood of the meridian which we refer to as the \emph{winding region}. 

Note that every intersection point $\mathbf{x}\in\intpt$ has $x_g=p$. Let $\mathbf{y}\in\intptt$. We say $\mathbf{y}$ is \emph{supported in the winding region} if $y_g$ is one of the $n$ points in the winding region. If $\mathbf{y}\in\intptt$ is supported in the winding region, then there is a ``closest'' point $\mathbf{x}\in\intpt$. Let $\b\theta$ be the top generator in homology of $\widehat{HF}(\#_{g-1}(S^1\times S^2))$. Then there is a canonical ``small triangle'' $\psi\in\pi_2(\mathbf{x},\b\theta,\mathbf{y})$ supported there. Recall the function 
\begin{equation*}
f(\mathbf{x})=\langle c_1(\mathfrak{s}_w(\psi)), [C]\rangle +[C]^2-2(n_w(\psi)-n_z(\psi))
\end{equation*}
defined by Ozsv\'ath and Szab\'o in \cite{OS-rationalsurgeries} where $\psi\in\pi_2(\mathbf{x},\theta,\mathbf{y})$ is a small triangle. Ozsv\'ath and Szab\'o show that this function depends only on $\mathbf{x}$. Their argument is revisited and expanded in \cite{Mark-Tosun}. In fact, by using Lemma \ref{periodicdom} we will see that this function actually computes the Alexander grading of the generator $\mathbf{x}$. 

First, we recall a few more facts and definitions from \cite{OS-rationalsurgeries}. The tuple $(\Sigma,\b\alpha,\b\beta,\b\gamma)$ forms a Heegaard triple-diagram that specifies a 4-manifold with three boundary components, $-Y$, $Y_{-n}(K)$ and $\dis\#_{g-1}(S^1\times S^2)$. Capping the final boundary component off with $\natural_{g-1} S^1\times B^3$ gives us a description of the cobordism $X_{-n}(K)$. In this 4-manifold, a triply periodic domain is a 2-chain $\mathcal{P}$ whose boundary is a sum of $\b\alpha$, $\b\beta$, and $\gamma_g$ curves. Given a surface in $X_{-n}(K)$ we can represent it via a triply periodic domain. 

Recall that for a triply periodic domain we also have a $c_1$-evaluation formula so that if $\psi$ is a Whitney triangle, 
\begin{equation}\label{c_1}
\langle c_1(\mathfrak{s}_w(\psi)), \mathcal{P}_\abc\rangle= \hat\chi(\mathcal{P}_\abc)+\#\partial \mathcal{P}_\abc+2\sigma(\psi,\mathcal{P}_\abc)-2n_w(\mathcal{P}_\abc)
\end{equation}
where $\sigma(\psi,\mathcal{P}_\abc)$ is the dual spider number \cite[Section 2.5]{OS-rationalsurgeries}. 

\begin{lemma} \label{Alexandergrading}
Let $(\Sigma,\b\alpha,\b\beta,w,z)$ be a doubly pointed Heegaard diagram for $K\subset Y$, and let $(\Sigma,\b\alpha,\b\beta,\b\gamma_n)$ be the related Heegaard diagram for the cobordism $X_{-n}(K)$ described above. Then for any $\mathbf{x}\in\intpt$ 
\begin{equation*}
f(\mathbf{x})=2A(\mathbf{x})
\end{equation*}
where $A(\mathbf{x})$ is the Alexander grading of $\mathbf{x}$.  
\end{lemma}

\begin{proof}
By Lemma \ref{periodicdom}, we can find a relative periodic domain $\mathcal{P}_{F}$ representing $[F]$ with 
\begin{equation*}
\partial \mathcal{P}_{F}=q\lambda_{\can}+cr\mu+q\alpha_g+\sum_{i=1}^{g-1} n_{\alpha_i}\alpha_i +\sum_{i=1}^{g-1} n_{\beta_i}\beta_i.
\end{equation*}
There is also a triply periodic domain $\mathcal{P}_{\abc}$ representing $[F_{qC}]$ with 
\begin{equation*}
\partial \mathcal{P}_{\abc}=q\lambda_n+(qn+cr)\mu+q\alpha_g+\sum_{i=1}^{g-1} n_{\alpha_i}\alpha_i +\sum_{i=1}^{g-1} n_{\beta_i}\beta_i.
\end{equation*}
We claim that 
\begin{equation} 
\langle c_1(\mathfrak{s}_{w,z}(\mathbf{x})), \mathcal{P}_{F}\rangle+q=\langle c_1(\mathfrak{s}_w(\psi)), \mathcal{P}_\abc\rangle-qn-cr- 2q(n_w(\psi)-n_z(\psi)).
\end{equation}

Since the right hand side only depends on $\mathbf{x}$, it suffices to prove the statement for the Whitney triangle $\psi$ with $n_w(\psi)=n_z(\psi)=0$. This follows by applying the first Chern class formulas as we show below. The diagram in Figure \ref{winding} shows the winding region for $-6$-surgery along $K$ and the multiplicities for the corresponding periodic domain $\mathcal{P}_\abc$.
\begin{figure}
\begin{center}
\includegraphics[scale=.9]{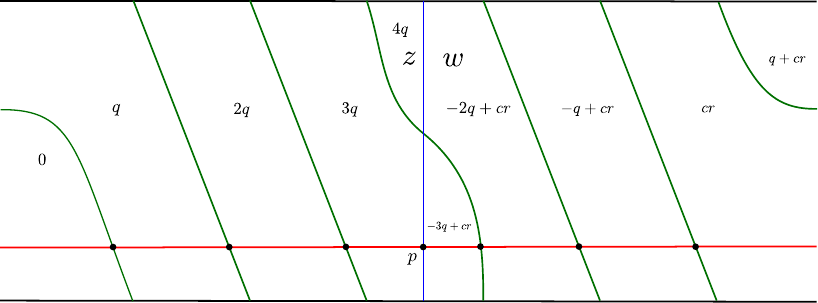}
\end{center}
\caption{The winding region for $-6$-surgery along $K$.} 
\label{winding}
\end{figure}

First consider the domain $\mathcal{P}_{F}$. Since $\mathbf{x}$ is of the form $\mathbf{x}=(x_1,\ldots, x_{g-1},p)$ we compute from Figure \ref{WR} that 
\begin{equation*}
n_{p}(\mathcal{P}_{F})=\frac{1}{4}(0+q+q+cr+cr)=\frac{2q+2cr}{4}=\frac{1}{2}(q+cr)
\end{equation*}
and that
\begin{equation*}
\bar n_{w,z}(\mathcal{P}_{F})=\frac{1}{2}(0+q)+\frac{1}{2}(cr+q+cr)=q+cr.
\end{equation*} 
Thus $2n_{p}(\mathcal{P}_{F})-\bar n_{w,z}(\mathcal{P}_{F})=0$. Applying Proposition \ref{Hedden-Levine} for the relative periodic domain $\mathcal{P}_F$ we see that
\begin{align*}
\langle c_1(\mathfrak{s}_{w,z}(\mathbf{x})), \mathcal{P}_{F}\rangle+q&=\hat\chi(\mathcal{P}_{F})+2n_\mathbf{x}(\mathcal{P}_{F})-\bar n_{w,z}(\mathcal{P}_{F})\\
&=\hat\chi(\mathcal{P}_{F})+2\sum_{i=1}^{g-1}n_{x_i}(\mathcal{P}_{F}).
\end{align*}

Now, fixing a point $\mathbf{y}=(y_1,\ldots,y_g)$ in the interior of $\psi$, we apply the $c_1$-evaluation formula of equation (\ref{c_1}) for the triply periodic domain $\mathcal{P}_{\abc}$. 

From Figure \ref{winding}, we calculate the dual spider number is $\sigma(\psi,\mathcal{P}_\abc)=\sum_{i=1}^gn_{y_i}$ and that the multiplicity at $w$ differs from the multiplicity at $y_g$ by $q$. Specifically, $n_{y_g}(\mathcal{P}_\abc)=n_w(\mathcal{P}_\abc)-q$. Therefore, $$2\sigma(\psi,\mathcal{P}_\abc)-2n_w(\mathcal{P}_\abc)= -2q+\sum_{i=1}^{g-1}n_{y_i}.$$
Thus,
\begin{align*}
\langle c_1(\mathfrak{s}_w(\psi)), \mathcal{P}_\abc\rangle&=\hat\chi(\mathcal{P}_\abc)+2q+qn+cr\\
&\:\:\;\;\;\;\;\;\:\:\:\:\:+\sum_{i=1}^{g-1}n_{\alpha_i}+\sum_{i=1}^{g-1}n_{\beta_i}+2\sum_{i=1}^{g-1}n_{y_i}-2q\\
&=\hat\chi(\mathcal{P}_\abc)+qn+cr+\sum_{i=1}^{g-1}(2n_{y_i}+n_{\alpha_i}+n_{\beta_i}).
\end{align*}

Outside of the winding region, we have that $2n_{x_i}(\mathcal{P}_{F})=2n_{y_i}(\mathcal{P}_\abc)+n_{\alpha_i}+n_{\beta_i}$. Therefore, putting this all together, along with the fact that $\hat\chi(\mathcal{P}_{F})=\hat\chi(\mathcal{P}_\abc)$ we obtain the result.
\end{proof}

Now we will use Lemma \ref{Alexandergrading} to interpret \cite[Theorem 4.1]{OS-rationalsurgeries}. However, here we phrase everything in terms of negative surgery. Let $\mathcal{C}_\mathfrak{s}=CFK^\infty(Y,K,\mathfrak{s})$.

\begin{theorem}\label{largesurgery}
Let $K\subset Y$ be knot of order $q$ in a rational homology sphere. Let $\mathfrak{t}_{\mathfrak{s},m}\in\S(X_{-n}(K))$ be the unique $\s$ structure on $X_{-n}(K)$ such that $\mathfrak{t}_{\mathfrak{s},m}|_{Y}=\mathfrak{s}$ and
\begin{equation*}
\langle c_1(\mathfrak{t}_{\mathfrak{s},m}),[F_{qC}]\rangle -nq-cr=2q(k_\mathfrak{s}+m).
\end{equation*} 
Let $\mathfrak{r}_m$ denote the restriction of $\mathfrak{t}_{m}^\mathfrak{s}$ to $Y_{-n}(K)$. Then for all sufficiently large $n$, the complex $CF^+(Y_{-n}(K), \mathfrak{r}_m)$ is isomorphic to $\mathcal{C}_{\mathfrak{s}}\{i=0 \text{ or } j=-m\}$ and there are maps $v^+$ and $\Phi^\mathfrak{s}_{-n,m}$ so that the following diagram commutes:   
\begin{equation*}
\xymatrix{\mathcal{C}_\mathfrak{s}\{i\geq 0\} \ar[r]^-{v+}\ar[d]& \mathcal{C}_\mathfrak{s}\{i=0 \text{ or } j=-m\}\ar[d]^{\Phi_{-n,m}^\mathfrak{s}}\\
CF^+(Y,\mathfrak{s})\ar[r]^-{F^\mathfrak{s}_{-n,m}}& CF^+(Y_{-n}(K), \mathfrak{r}_m).\\}
\end{equation*}
\end{theorem}
\begin{proof}

Consider the map $$\Phi_{-n,m}^{\mathfrak{s}}: CFK^\infty(Y,K,\mathfrak{s})\to CF^\infty(Y_{-n}(K), \mathfrak{r}_m)$$ defined by $$\Phi_{-n,m}^{\mathfrak{s}}([\mathbf{x},i,j])=\sum_{\mathbf{y}\in\intptt}\sum_{\substack{\psi\in\pi_2(\mathbf{x},\b\theta,\mathbf{y})\\ n_w(\psi)-n_z(\psi)=i-j+m} }\#\widehat{\mathcal{M}}(\psi)\cdot[y,i-n_w(\psi)].$$

The proof now follows from the argument of \cite[Theorem 4.1]{OS-knotinvariants} and applying Lemma \ref{Alexandergrading}.
\end{proof}

\subsection{A four-dimensional interpretation of $\tau_\mathfrak{s}(Y,K)$}
The large surgery theorem now allows us to relate $\tau_\mathfrak{s}(Y,K)$ to the map induced on Floer homology by the cobordism $X_{-n}(K)$.  Let $K$ be a knot of order $q$ in a rational homology sphere $Y$. Let
\begin{equation*}
\widehat{F}_{-n,m}^\mathfrak{s}: \widehat{HF}(Y,\mathfrak{s})\to \widehat{HF}(Y_{-n}(K),\mathfrak{r}_m)
\end{equation*}	
be the map induced by the 4-manifold cobordism $X_{-n}(K)$ from $Y$ to $Y_{-n}(K)$, where $\mathfrak{r}_m$ is the restriction to $Y_{-n}(K)$ of the unique $\s$-structure $\mathfrak{t}_m$ on $X_{-n}(K)$ satisfying $\mathfrak{t}_m|_Y=\mathfrak{s}$ and 
\begin{equation*}
\langle c_1(\mathfrak{t}_m),[F_{qC}]\rangle -nq-cr=2q(k_\mathfrak{s}+m).
\end{equation*} 

\begin{proposition}\label{4-dimtau}
Let 
\begin{equation*}
\mathcal{S}=\im(\rho_*)\cap\im(\pi_*)\subset HF^+(Y,\mathfrak{s}).
\end{equation*}
Then for all $|n|$ sufficiently large, we have the following: 
\begin{itemize}
\item If $k_\mathfrak{s}+m<\tau_{\mathfrak{s}}(Y,K)$, then for all $\beta\in\rho_*^{-1}(\mathcal{S})$, we have $\widehat{F}_{-n,m}^\mathfrak{s}(\beta)\neq 0$.
\item If $k_\mathfrak{s}+m>\tau_{\mathfrak{s}}(Y,K)$ then there exists $\beta\in\rho_*^{-1}(\mathcal{S})$ such that $\widehat{F}_{-n,m}^\mathfrak{s}(\beta)=0$. 
\end{itemize} 
\end{proposition}
\begin{proof} We adapt the argument given in \cite{OS-4ball}. 
Let $\mathcal{C}=CFK^\infty(Y,K,\mathfrak{s})$. Note that the diagram
\begin{equation*}
\begin{small}
\xymatrix{
0\ar[r]&\mathcal{C}_{\{i=0,j\leq m\}}\cong \mathcal{F}_{\mathfrak{s},m}\ar[r]^{i_m}\ar[d]&\mathcal{C}_{\{i=0\}}\cong \widehat{CF}(Y,\mathfrak{s})\ar[r]^{p_m}\ar[d]^{\widehat{f}}& \mathcal{C}_{\{i=0,j>m\}}\cong \mathcal{Q}_{\mathfrak{s},m} \ar[r]\ar[d]&0\\
0\ar[r]&\mathcal{C}_{\{i\geq 0, j=m\}}\ar[r]&\mathcal{C}_{\{\min(i,j-m)=0\}}\ar[r]& \mathcal{C}_{\{i=0,j>m\}}\ar[r]&0\\
}
\end{small}
\end{equation*}
commutes. By Theorem \ref{largesurgery}, when $|n|$ is sufficiently large, we can identify the sub-quotient complex $\mathcal{C}_{\{\min(i,j-m)=0\}}$ with $\widehat{CF}(Y_{-n}(K),\mathfrak{r}_m)$ and the map induced by $\widehat{f}$ on homology with the map $\widehat{F}_{-n,m}^\mathfrak{s}$. To prove the first part, assume that $k_{\mathfrak{s}}+m< \tau_{\mathfrak{s}}(Y,K)$. Let $\gamma\in\rho^{-1}_*(\mathcal{S})$. By the assumption, $\gamma$ is not in the image of $I_m$. Thus, by exactness, $P_m(\gamma)\neq 0$ and in particular, since the diagram commutes, $\widehat{F}_{-n,m}^\mathfrak{s}(\gamma)\neq 0$. 

On the other hand, the map $\widehat{f}$ factors through $p_{m-1}$. Thus, if $k_{\mathfrak{s}}+m>\tau_{\mathfrak{s}}(Y,K)$, then there exists an element $\gamma\in\rho_*^{-1}(\mathcal{S})$ such that $\gamma\in\im(I_{m-1})$. Thus, $P_{m-1}(\gamma)=0$ and by commutativity, $\widehat{F}_{-n,m}(\gamma)=0$ as well. 
\end{proof}

\section{Genus bounds}
Consider a negative definite 4-manifold $W$ with boundary a rational homology sphere $Y$. By removing a small ball from W, we obtain a 4-manifold which is a cobordism $W-B^4$ from $S^3$ to $Y$. In order to obtain our genus bounds, we need to ensure that the map on Floer homology induced by this cobordism is nontrivial. 

\begin{definition}
Let $W$ be a negative definition 4-manifold with boundary $Y$, a rational homology sphere. A $\s$-structure $\mathfrak{t}$ on $W$ is \emph{sharp} if $$c_1^2(\mathfrak{t})+b_2(W-B^4)=4d(Y,\mathfrak{t}|_Y).$$
\end{definition}

\begin{lemma}\label{lem:bound1} Let $W$ be a negative definite 4-manifold with boundary $Y$ a rational homology sphere. Then for any sharp $\s$-structure $\mathfrak{t}$ on $W$ the map 
\begin{equation*}
\widehat{F}_{W-B^4,\mathfrak{t}}:\widehat{HF}(S^3)\to \widehat{HF}(Y,\mathfrak{s})
\end{equation*}
is nontrivial. Moreover, $\widehat{HF}(S^3)$ is generated by a single element which maps to a non-torsion element of $HF^+(Y,\mathfrak{s})$ under the composition 
\begin{equation*}
\widehat{HF}(S^3)\to \widehat{HF}(Y,\mathfrak{s})\to HF^+(Y,\mathfrak{s}). 
\end{equation*}
\end{lemma}

\begin{proof}
Considering $X=W-B^4$ as a cobordism from $S^3$ to $Y$, we know that for each $\s$-structure on $X$, the map $$F^\infty_{X,\mathfrak{t}}:HF^\infty(S^3)\to HF^\infty(Y,\mathfrak{t}|_{Y})$$ is an isomorphism and the shift in degree can be calculated using the formula $$\frac{c^2_1(\mathfrak{t})-2\chi(X)-3\sigma(X)}{4}=\frac{c_1^2(\mathfrak{t})+b_2(X)}{4}.$$
The cobordism $X$ gives rise to a commutative square
\begin{equation*}
\xymatrix{
\widehat{HF}(S^3)\ar[r]^-{\widehat{F}_{X,\mathfrak{t}}}\ar[d]&\widehat{HF}(Y,\mathfrak{t}|_{Y})\ar[d]\\
HF^+(S^3)\ar[r]^-{F^+_{X,\mathfrak{t}}}&HF^+(Y,\mathfrak{t}|_{Y})
}
\end{equation*}
The top map $\widehat{F}_{X,\mathfrak{t}}$ is nontrivial provided that the element of $HF^+(S^3)$ in grading zero maps to a minimally graded element of $HF^+(Y,\mathfrak{s})$, that is, an element of grading $d(Y,\mathfrak{t}|_{Y})$. Since $\mathfrak{t}$ is a sharp $\s$-structure, we know that $d(Y,\mathfrak{t}|_{Y})$ is exactly equal to the grading shift.

\end{proof}

In addition, we need the following lemma which is Lemma 3.5 from \cite{OS-4ball}. While, the statement given in {\cite{OS-4ball}} claims a stronger bound, the bound given here is the correct one. The corrected statement and proof of Lemma \ref{os-bound} below is due to both the author and Matthew Hedden.

\begin{lemma}[{\cite[Lemma 3.5]{OS-4ball}}]\label{os-bound}
Let $N$ be the total space of a disk bundle with Euler number $-n<0$ over an oriented two manifold $\Sigma$ of genus $g$. If $n\geq 2g-1$ then the map 
\begin{equation*}
\widehat{F}_{N-B,\mathfrak{s}}:\widehat{HF}(S^3)\to\widehat{HF}(\partial N, \mathfrak{t}|_{\partial N})
\end{equation*}
is trivial whenever
\begin{equation*}
\langle c_1(\mathfrak{t}),[\Sigma]\rangle+[\Sigma]\cdot[\Sigma]>2g(\Sigma).
\end{equation*}
\end{lemma}

\begin{proof}
The disk bundle $N$ has a handle decomposition with one 0-handle, $2g$ 1-handles and a single 2-handle. Specifically, the 0-handle and 1-handles give a handle decomposition for $\natural_{2g} (S^1\times B^3)$ and the 2-handle is added with $(-n)$-framing along $B_{g}$ the ``Borromean knot". Thus, $N-B^4=W_1\cup W_2$, where $W_1=\natural_{2g}(S^1\times B^3)-B^4$ and $W_2$ is the 2-handle addition along $B_g$. 

Let $\mathfrak{t}|_{W_i}=\mathfrak{t}_i$. Then $\mathfrak{t}=\mathfrak{t}_1\#\mathfrak{t}_2$ and the map $\widehat{F}_{N-B^4,\mathfrak{t}}$ factors as $$\widehat{HF}(S^3)\xrightarrow{\widehat{G}_{W_1,\mathfrak{t}_1}} \widehat{HF}(\#_{2g}(S^1\times S^2))\xrightarrow{\widehat{G}_{W_2,\mathfrak{t}_2}} \widehat{HF}(\partial N,\mathfrak{t}|_{\partial N}).$$

In Section 9 of \cite{OS-knotinvariants} Ozsv\'ath and Szab\'o compute the knot Floer homology of $B_g$. There they show that in Alexander grading $j$, $$\widehat{HFK}(\#_{2g}S^1\times S^2, B_g, j)=\Lambda^{g+j}H^1(\Sigma).$$ They also show that $$CFK^\infty(\#_{2g}S^1\times S^2, B_g)=\Lambda^{*}H^1(\Sigma)\otimes \F_2[U,U^{-1}]$$ and has no differentials. 

The map $$\widehat{G}_{W_1,\mathfrak{t}_1}:\widehat{HF}(S^3)\to \widehat{HF}(\#_{2g}S^1\times S^2)$$ increases the homological grading by $g$. Therefore, the generator of $\widehat{HF}(S^3)$ is sent to the top graded element of $\widehat{HF}(\#_{2g}S^1\times S^2)$.

Now we focus on the second map $$\widehat{G}_{W_2,\mathfrak{t}_2}:\widehat{HF}(\#_{2g}(S^1\times S^2))\to\widehat{HF}(\partial N, \mathfrak{t}|_{\partial N})$$ induced by the 2-handle addition. Let $\langle c_1(\mathfrak{t}),[\Sigma]\rangle +[\Sigma]^2=2k$. The large surgery theorem implies that when $n\geq 2g(\Sigma)-1$, the sub-quotient complex $\mathcal{C}\{\min(0,j-k)=0\}$ can be identified with $\widehat{CF}(\partial N, \mathfrak{t}|_{\partial N})$. In addition, the map  $\widehat{G}_{W_2,\mathfrak{t}_2}$ can be computed from looking at from $CFK^\infty(\#_{2g}S^1\times S^2, B_g)$. In particular, this map is trivial whenever $k>g(\Sigma)$. 

On the other hand, if $\langle c_1(\mathfrak{t}),[\Sigma]\rangle+[\Sigma]\cdot[\Sigma]=2g(\Sigma)$ the map $\widehat{G}_{W_2,\mathfrak{t}_2}$ is non-trivial on the top graded element of $\widehat{HF}(\#_{2g}S^1\times S^2)$. Therefore, the map $\widehat{F}_{N-B,\mathfrak{t}}$ will be nontrivial as well. 

\end{proof}

Now consider a knot $K$ in $Y=\partial W$. Adding a 2-handle to $W$ along $K$ with $(-n)$-framing, we form a new 4-manifold $W_{-n}(K)$ which decomposes as $W\cup_Y X_{-n}(K)$ where $X_{-n}(K)$ is the 2-handle cobordism from $Y$ to $Y_{-n}(K)$. If $\Sigma$ is a surface in W with boundary $K$, we can form a closed surface $\widehat\Sigma=\Sigma \cup C$ in $H_2(W_{-n}(K))$. Consider the Mayer-Vietoris sequence, 
\begin{equation*}
0\to H_2(W)\oplus H_2(X_{-n})\xrightarrow{i_*} H_2(W_{-n}(K))\xrightarrow{j_*} H_1(Y)\to\ldots.
\end{equation*}
 While $[\widehat\Sigma]$ is not in the kernel of $j$, $q[\widehat\Sigma]$ is in the kernel of $j$. Thus, $q[\widehat\Sigma]$ splits as a class in $H_2(W)\oplus H_2(X_{-n}(K))$. We can think of this splitting geometrically as
\begin{equation}\label{classes}
q[\widehat\Sigma]=i_*([q\Sigma\cup -F]\oplus[F_{qC}])
\end{equation}		
where $F$ is a rational Seifert surface for $K$.

\begin{theorem}
Let $W$ be a negative definite 4-manifold and $K$ a knot in $\partial W=Y$. Let $F$ be a $q$-Seifert surface for $K$ and $\mathfrak{t}$ a sharp $\s$-structure on $W$. Then for any surface $\Sigma$ such that $\partial\Sigma=K$, 
\begin{equation*}
\frac{1}{q}\langle c_1(\mathfrak{t}),[q\Sigma\cup -F]\rangle+\frac{1}{q^2}[q\Sigma\cup -F]^2+2\tau_\mathfrak{s}(Y,K)\leq 2g(\Sigma)
\end{equation*}
where $\mathfrak{s}=\mathfrak{t}|_{Y}$.
\end{theorem}	

\begin{proof}

Consider our decomposition of the 4-manifold $W_{-n}(K)$ as 
\begin{equation*}
W_{-n}(K)=W\cup_Y\ X_{-n}(K).
\end{equation*}
If we remove a small ball $B$ from the interior of $W_{-n}(K)$ we can view $W_{-n}(K)-B$ as a cobordism from $S^3$ to $Y_{-n}(K)$ which is the composition of the cobordisms $W-B$ from $S^3$ to $Y$ and $X_{-n}(K)$ from $Y$ to $Y_{-n}(K)$.  

Since $\mathfrak{t}$ is sharp, by Lemma \ref{lem:bound1}, the map 
\begin{equation*}
F_{W-B,\mathfrak{t}}:\widehat{HF}(S^3)\to \widehat{HF}(Y,\mathfrak{s})
\end{equation*}
takes the generator of $\widehat{HF}(S^3)$ to an element $\gamma\in \rho_*^{-1}(\mathcal{S})$. 

Fix $m<\tau_{\mathfrak{s}}(Y,K)-k_\mathfrak{s}$ and choose $|n|$ large enough that Proposition \ref{4-dimtau} and Lemma \ref{os-bound} hold. Let $\mathfrak{t}_m$ be the $\s$-structure on $X_{-n}(K)$ such that $\mathfrak{t}|_Y=\mathfrak{s}$ and 
\begin{equation*}
\langle c_1(\mathfrak{t}_m),[F_{qC}]\rangle +[F_{qC}]^2=2q(k_\mathfrak{s}+m).
\end{equation*}
Since $2q(k_\mathfrak{s}+m)\leq 2q\tau_\mathfrak{s}(Y,K)$, the map $F_{n,m}^\mathfrak{s}:\widehat{HF}(Y,\mathfrak{s})\to \widehat{HF}(Y_{-n}(K),\mathfrak{r}_m)$ applied to $\gamma$ is nontrivial and therefore, the composition 
\begin{equation*}
F_{n,m}^\mathfrak{s}\circ F_{W-B,\mathfrak{t}}: \widehat{HF}(S^3) \to \widehat{HF}(Y_{-n}(K),\mathfrak{r}_m)
\end{equation*}
is nontrivial as well.

On the other hand, since $W_{-n}(K)-B$ is a 4-manifold, we are free to factor this cobordism through any intervening 3-manifolds. Decompose $W_{-n}(K)$ as $W_1\cup W_2$ where $W_1=\nu(\widehat{\Sigma})$ is the tubular neighborhood of $\widehat{\Sigma}$ and $W_2$ is the complement of this tubular neighborhood.

Since we have chosen $n$ and $m$ so that the composition $F_{n,m}^\mathfrak{s}\circ F_{W-B,\mathfrak{t}}$ is nontrivial, the map $$\widehat{HF}(S^3)\to\widehat{HF}(\partial\nu(\widehat{\Sigma}),\hat{\mathfrak{t}}|_{\partial\nu(\widehat{\Sigma})})$$ where $\hat{\mathfrak{t}}=\mathfrak{t}\#\mathfrak{t}_m$ is nontrivial as well. Therefore by Lemma \ref{os-bound} we have 
\begin{equation*}
\langle c_1(\hat{\mathfrak{t}}),[\widehat{\Sigma}]\rangle +[\widehat{\Sigma}]^2\leq 2g(\widehat{\Sigma}).
\end{equation*}

Applying equation (\ref{classes}) to rewrite the left-hand side shows that 
\begin{align*}
\langle c_1(\hat{\mathfrak{t}}),[\widehat{\Sigma}]\rangle +[\widehat{\Sigma}]^2 
&=\langle c_1(\hat{\mathfrak{t}}),i_*([q\Sigma\cup -F]\oplus[F_{qC}])\rangle +\frac{1}{q^2}  ([q\Sigma\cup -F]\oplus[F_{qC}])^2\\
&=\frac{1}{q}\langle c_1(\mathfrak{t}),[q\Sigma\cup -F]\rangle+\frac{1}{q^2}[q\Sigma\cup -F]^2\\&\quad\quad\quad\quad+\frac{1}{q}\langle c_1(\mathfrak{t}_{m}^{\mathfrak{s}}),[F_{qC}]\rangle +\frac{1}{q^2}[F_{qC}]^2\\
&=\frac{1}{q}\langle c_1(\mathfrak{t}),[q\Sigma\cup-F]\rangle+\frac{1}{q^2}[q\Sigma\cup -F]^2 +2(k_{\mathfrak{s}}+m).
\end{align*}
Thus, 
\begin{equation}
\frac{1}{q}\langle c_1(\mathfrak{t}),[q\Sigma\cup- F]\rangle+\frac{1}{q^2}[q\Sigma\cup -F]^2+2(k_{\mathfrak{s}}+m)\leq 2g(\widehat{\Sigma})=2g(\Sigma).
\label{bound1}
\end{equation}
Since equation (\ref{bound1}) holds for all integers $m<\tau_\mathfrak{s}(Y,K)-k_\mathfrak{s}$ we may take $k_{\mathfrak{s}}+m= \tau_{\mathfrak{s}}(Y,K)-1$.
Making this substitution and simplifying we obtain,   
\begin{equation*}
\frac{1}{q}\langle c_1(\mathfrak{t}),[q\Sigma\cup -F]\rangle+\frac{1}{q^2}[q\Sigma\cup -F]^2+2\tau_\mathfrak{s}(Y,K)-2\leq 2g(\Sigma)
\label{longequation} 
\end{equation*}

Finally, we can improve this bound slightly by exploiting the additivity of the terms. This final part of the argument is due to the author and Matthew Hedden. Let $W_d=\natural_d W$. Then $W_d$ is a negative definite 4-manifold with boundary $\#_dY$. In addition the knot $\#_dK\subset \#_d Y$, bounds the surface $\natural_d\Sigma$ and has a $q$-Seifert surface $\natural_d F$. Therefore, we can calculate that 
\begin{align*}
\langle c_1(\#_d\mathfrak{t}),[q\natural_d \Sigma\cup-\natural_d F]\rangle &= d\langle c_1(\mathfrak{t}),[q\Sigma\cup -F]\rangle\\
\frac{1}{q^2}[q\natural_d\Sigma\cup\natural_d F]^2&=\frac{d}{q^2}[q\Sigma\cup -F]^2\\
\tau_{\#_d\mathfrak{s}}(\#_d Y,\#_d K)&=d\tau_{\mathfrak{s}}(Y,K)\\
g(\natural_d\Sigma)&=dg(\Sigma)
\end{align*}
Applying Equation \ref{longequation} to $\#_dK$ in $X_d$, and dividing through by $d$, we get that 
$$\frac{1}{q}\langle c_1(\mathfrak{t}),[q\Sigma\cup -F]\rangle+\frac{1}{q^2}[q\Sigma\cup -F]^2+2\tau_\mathfrak{s}(Y,K)\leq 2g(\Sigma)+\frac{2}{d}$$
holds for any choice of $d$. Thus,
$$\frac{1}{q}\langle c_1(\mathfrak{t}),[q\Sigma\cup- F]\rangle+\frac{1}{q^2}[q\Sigma\cup -F]^2+2\tau_\mathfrak{s}(Y,K)\leq 2g(\Sigma).$$
\end{proof}

\begin{corollary}\label{corollary}
Let $W$ be a rational ball with $\partial W=Y$, and $\Sigma$ a properly embedded surface in $W$ with $\partial\Sigma=K$. Then for each $\s$-structure $\mathfrak{s}$ on $Y$ that extends over $W$
\begin{equation*}
|\tau_\mathfrak{s}(Y,K)|\leq g(\Sigma).
\end{equation*}
\end{corollary}
\begin{proof}
Applying Theorem \ref{maintheorem} we see that $\tau_\mathfrak{s}(Y,K)\leq g(\Sigma)$. On the other hand, reversing the orientation of $W$ implies that $\tau_{\mathfrak{s}}(-Y,K)\leq g(\Sigma)$. The result now follows by applying Lemma \ref{properties} (1). \end{proof}

\section{Examples}
\subsection{A slice knot and a non-slice knot in $L(4,1)$}
\begin{figure}
\begin{center}
\includegraphics[scale=.2]{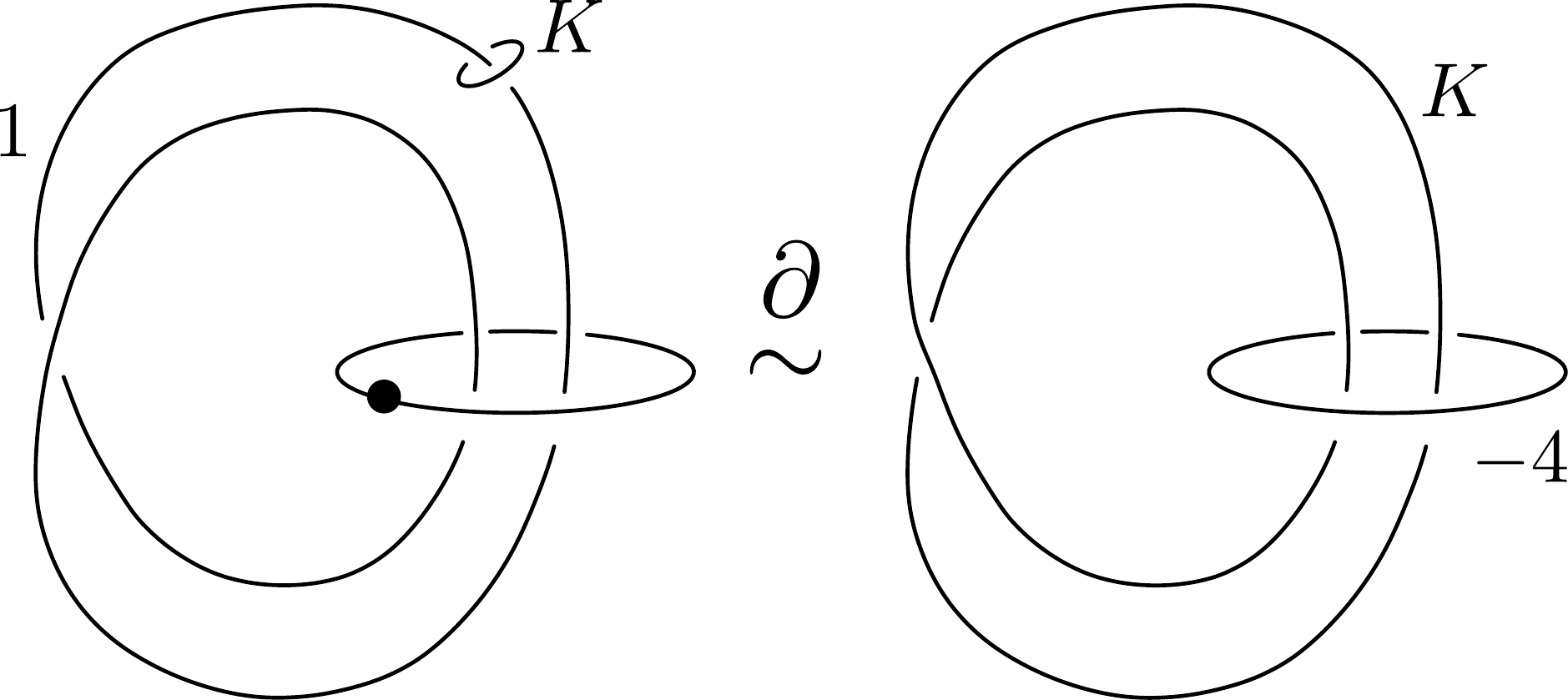}
\end{center}
\caption{The knot $K$ is slice in a rational ball with boundary $L(4,1)$.}
\label{QH4ball}
\end{figure} 
The lens space $L(4,1)$ bounds a rational homology 4-ball $W$ as shown in the Kirby diagram in Figure \ref{QH4ball} and the map 
\begin{align*}
H_1(L(4,1))& \to H_1(W)\cong \Z_2
\end{align*}
is reduction modulo two. Thus, any knot of order two in $H_1(L(4,1))$ will be null-homologous in $W$ and applying Corollary \ref{corollary} to any such knot gives us a bound on the genus of a surface bounded by the knot. 

Let $K$ be the knot of order two in $L(4,1)$ shown in Figure \ref{QH4ball}. This knot clearly bounds a disk in $W$. Therefore, if we calculate the $\tau$-invariants associated to $K$, there must be at least two that vanish corresponding to the $\s$-structures that extend over $W$. 
 
A genus one Heegaard diagram for $K$ is shown on the left in Figure \ref{HD}. From this Heegaard diagram we can calculate the Alexander grading of each element by finding a periodic domain for a rational Seifert surface $F$ and using the $c_1$-evaluation formula from Proposition \ref{Hedden-Levine}. The Alexander gradings are listed in the table in Figure \ref{AlexandergradingsK}.

\begin{figure}
\begin{center}
\includegraphics[scale=.3]{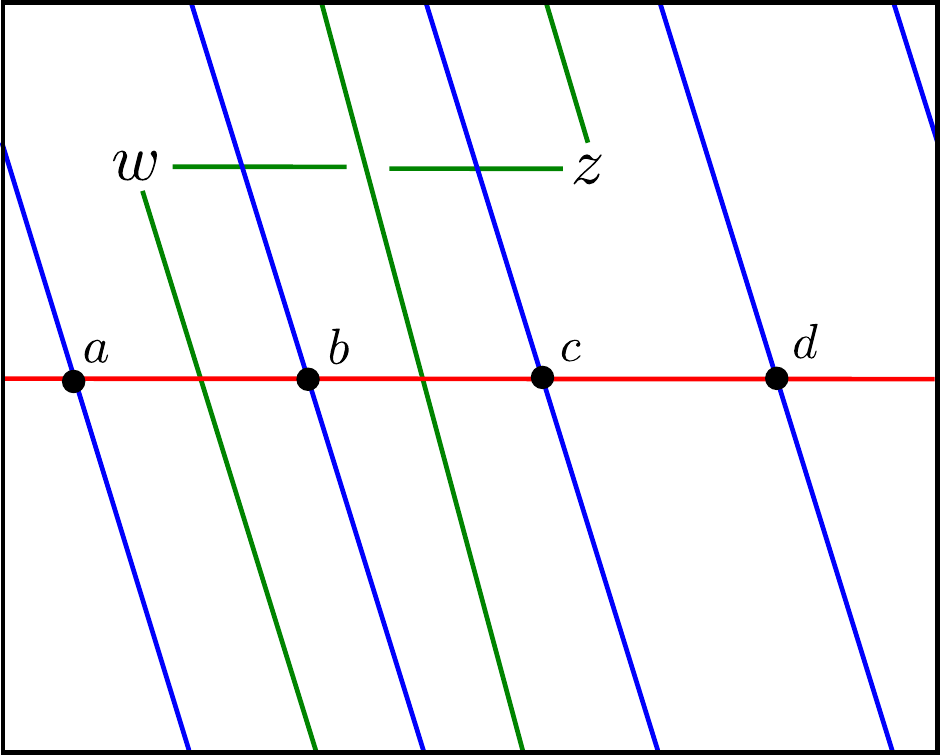}\hspace{.5cm}
\includegraphics[scale=.3]{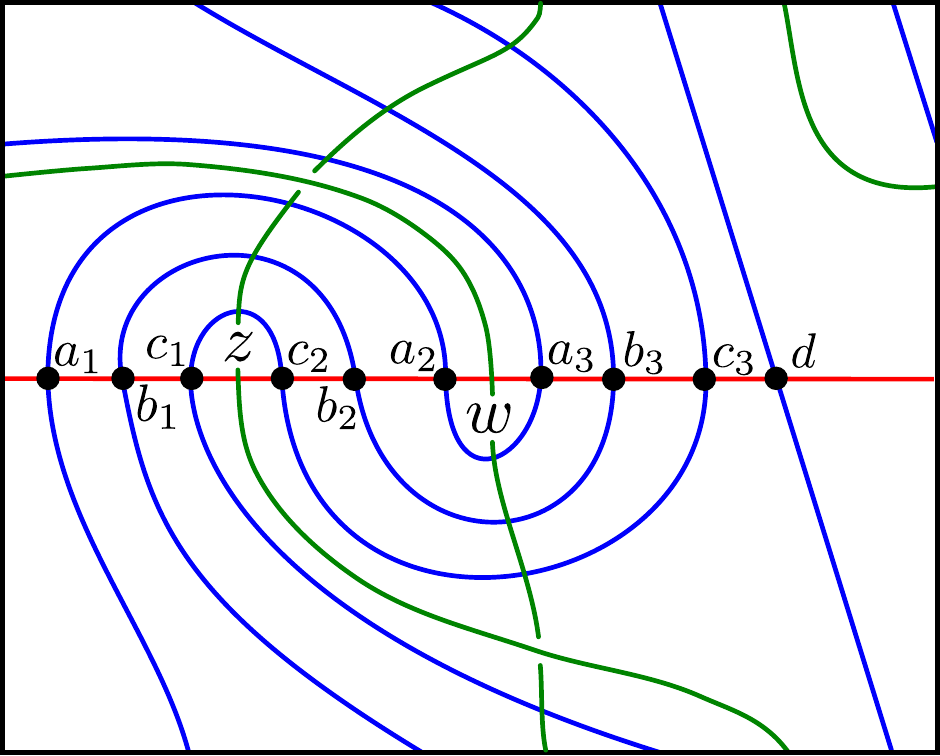}
\caption{Heegaard diagrams for $K$ and $J$ in $L(4,1)$ with undercrossing pushed slightly into the $U_{\b\alpha}$ handlebody.}
\label{HD}
\end{center}
\end{figure}

Each intersection point in the diagram represents a distinct $\s$-structure, so we denote $\mathfrak{s}_w(a)$ by $\mathfrak{s}_a$ and denote the rest of the $\s$-structures accordingly. Since there are no holomorphic disks in this diagram, the Alexander grading of each element is equal to $\tau_{\mathfrak{s}_w(x)}(L(4,1), K)$. Thus, $\mathfrak{s}_b$ and $\mathfrak{s}_d$ extend over $W$. 

\begin{table}
\begin{center}  
\setlength{\extrarowheight}{2 pt}
\begin{tabular}{|c|c|c|c|c|}
\hline
$x$& $a$&$b$&$c$&$d$ \\[0.5 ex]
\hline 
$A(x)$&$-\frac{1}{2}$&0& $\frac{1}{2}$&0\\[1 ex]
\hline
\end{tabular}
\end{center}
\caption{Alexander gradings for generators of $\widehat{CFK}(Y,K)$.}
\label{AlexandergradingsK}
\end{table}

Now consider the knot $J$ given by the Heegaard diagram on the right in Figure \ref{HD}. This knot is also order two in $H_1(L(4,1))$ and Table \ref{J} shows the Alexander grading of each generator. 

Applying Equation (\ref{3genus}), the rational 3-genus of the knot $J$ is $||J||=1$. This implies that if $F$ is a connected rational Seifert surface for $J$, then $-\chi(F)\geq 4$. Since a connected Seifert surface will have either one or two boundary components, this means that we must have $g(F)\geq 2$.

\begin{table}
\begin{center}
\setlength{\extrarowheight}{2 pt}
\begin{tabular}{|c|c|c|c|c|c|c|c|c|c|c|}
\hline
$x$&$a_1$&$a_2$&$a_3$&$b_1$&$b_2$&$b_3$&$c_1$&$c_2$&$c_3$&$d$\\[0.5 ex]
\hline
$A(x)$&$\frac{1}{2}$&$-\frac{1}{2}$&$-\frac{3}{2}$& $1$& $0$ & $-1$ &$\frac{3}{2}$& $\frac{1}{2}$&$-\frac{1}{2}$ & 0 \\[0.5 ex]
\hline
\end{tabular}
\end{center} 
\caption{Alexander gradings for generators of $\widehat{CFK}(Y,J)$.}
\label{J}
\end{table}
We can also determine the structure of $CFK^\infty(L(4,1),J,\mathfrak{s})$ in each $\s$-structure from the Heegaard diagram. We have the following $\s$-equivalence classes for the generators:
\begin{align*}
\mathfrak{s}_a&=\mathfrak{s}_w(a_1)=\mathfrak{s}_w(a_2)=\mathfrak{s}_w(a_3)\\
\mathfrak{s}_b&=\mathfrak{s}_w(b_1)=\mathfrak{s}_w(b_2)=\mathfrak{s}_w(b_3);\\
\mathfrak{s}_c&=\mathfrak{s}_w(c_1)=\mathfrak{s}_w(c_2)=\mathfrak{s}_w(c_3)\\
\mathfrak{s}_d&=\mathfrak{s}_w(d).\\
\end{align*}
Figure \ref{CFK} shows $CFK^\infty(L(4,1),J,\mathfrak{s}_b)$. Thus, the $\tau$-invariants are given by,
\begin{figure}[b]
\begin{tikzpicture}
\draw[<->] (0,2)--(0,-2);
\draw[<->] (2,0)--(-2,0);
\foreach \x in {-2,...,1}
	\filldraw[fill=black](\x,\x) circle (2pt);
\foreach \x in {-2,...,1}
	\draw[->, thick](\x,\x+1)--(\x,\x);
\foreach \x in {-1,...,2}
	\filldraw[fill=black](\x-1,\x) circle (2pt);
\foreach \x in {-2,...,1}
	\draw[->, thick](\x+1,\x)--(\x,\x);
\foreach \x in {-2,...,1}
	\filldraw[fill=black](\x+1,\x) circle (2pt);
\end{tikzpicture}
\caption{$CFK^\infty(L(4,1),J,\mathfrak{s}_b).$}
\label{CFK}
\end{figure}
\begin{align*}
\tau_{\mathfrak{s}_a}(L(4,1),J)&=-\frac{3}{2}; &\tau_{\mathfrak{s}_b}(L(4,1),J)&=-1;\\
\tau_{\mathfrak{s}_c}(L(4,1),J)&=-\frac{1}{2}; &\tau_{\mathfrak{s}_d}(L(4,1),J)&=0.\\
\end{align*}
By our previous calculation for $\tau_{\mathfrak{s}}(L(4,1),K)$, $\mathfrak{s}_b$ and $\mathfrak{s}_d$ both extend over $W$. Thus, since $\tau_{\mathfrak{s}_b}(L(4,1),J)=1$, $J$ is not slice. Note that the bound we obtain is lower than the $3$-dimensional genus bound coming from Equation (\ref{3genus}). In addition, by a remark of Celoria \cite[Remark 17]{Celoria}, since $\tau_{\mathfrak{s}_d}(L(4,1),J)=0$, we also know that $J$ cannot be rationally concordant to a connected sum of knots $K\# K'$ for $K'$ in in $S^3$ with $\tau(K')=1$. 

\subsection{Calculations in non-$L$-space examples}

The previous example relied on the fact that lens spaces have nice genus one Heegaard diagrams. To compute $\tau$-invariants for knots in non-$L$-space examples, we will apply the following theorem of Truong. 

Let $K$ be a knot in the 3-sphere and consider the 3-manifold obtained by $n$-surgery along $K$. We are interested in the knot $u_p$ that is the $(p,1)$ cable of the meridian of $K$ viewed as a knot in $S^3_n(K)$. 

\begin{theorem}[\cite{Truong}]\label{Truong}
Let $K\subset S^3$ be a knot and fix $m,p\in\Z$.Then there exists $N=N(m,p)>0$ such that for all $n>N$, the $\Z$-filtration on $\widehat{CF}(S^3_n(K),\mathfrak{s}_m)$ induced by $\mu_p$ is isomorphic to the filtered chain homotopy type of the $(p+1)$-step filtration $\mathcal{F}$ on $\mathcal{C}\{\max(i,j-m)=0\}$ given by: 
\begin{equation*}
0\subset\mathcal{C}\{i<-p+1,j=m\}\subset \ldots \subset \mathcal{C}\{i<0,j=m\}\subset \mathcal{C}\{\max(i,j-m)=0\}.
\end{equation*}
\end{theorem}

\begin{figure}
\includegraphics[scale=.7]{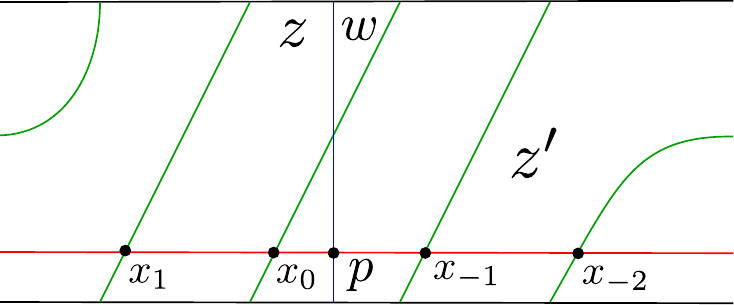}
\caption{}
\label{WindingRegion4}
\end{figure}

Theorem \ref{Truong} is a refinement of the large surgery theorem. Given a Heegaard diagram for $(S^3,K)$ a diagram for $(S^3_n(K),\mu_p)$ is obtained by adding a third basepoint $z'$, $p$-regions to the right of $\beta_g=\mu$. Figure \ref{WindingRegion4} shows a the winding region for $(S^3_4(K), \mu_2)$. 

For a fixed $p$, we further refine Theorem \ref{Truong} in the following way:  
\begin{proposition}\label{refinement}
Let $K$ be a knot in $S^3$ and fix $p\in \Z$. If $n\geq 2g(K)+p-1$, then the $\Z$-filtration on $\widehat{CF}(S^3_n(K),\mathfrak{s}_m)$ induced by $\mu_p$ is isomorphic to the filtered chain homotopy type of the $(p+1)$-step filtration $\mathcal{F}$ on $\mathcal{C}\{\max(i,j-m)=0\}$ described above for each $\lfloor\frac{-n}{2}\rfloor\leq m\leq  \lfloor\frac{n}{2}\rfloor -1$. 
\end{proposition}
\begin{proof} This follows from the argument in Theorem 4.3 of \cite{Hedden-Whitehead} and its refinement in Theorem 4.2 of \cite{Hedden-Kim-Livingston}. There Hedden recalls that for a knot $K$ in $S^3$ and an integer $n\geq 2g(K)$ we can find a Heegaard diagram that has at least one $\s$-equivalence class such that every generator in that class is supported in the winding region. Employing the technique of ``moving the basepoint" we can change the particular $\s$-structure represented by this $\s$-equivalence class. In the present setting,``moving the basepoint" means moving the points $w$ and $z'$ in the diagram for $(S^3_n(K),\mu_k)$. Since $w$ and $z'$ must remain $p$ regions apart, we can only move the basepoint $n-p+1$ times. If $n=2g(K)$, this means we have only represented $n-p+1$ distinct $\s$-structures. Increasing $n$ to at least $2g(K)+p-1$, we have a diagram in which at least $p$ distinct $\s$-equivalence classes are supported only in the winding region. Moving the basepoint through this diagram, we can ensure that each of the $2g(K)+p-1$ $\s$-structures is represented. 
\end{proof}

Proposition \ref{refinement} allows us to recover the \emph{relative} Alexander filtration induced by $\mu_p$ in $S^3_n(K)$ for sufficiently large $n$. To recover the absolute Alexander grading, we examine generators supported in the winding region. For our example, we focus on the case of $\mu_2$. However, similar calculations can be used to recover the Alexander grading for any $\mu_p$. 

\begin{lemma}\label{Alexgradings} 
Let $[{\bf y}, i ,j]$ be a generator of  the sub-quotient complex $\mathcal{C}\{\max(i,j-m)=0\}$  in $CFK^\infty(S^3,K)$. Viewing $[{\bf y}, i ,j]$ as an element of $\widehat{CFK}(S^3_n(K),\mu_2)$, if $n$ is sufficiently large, then its Alexander grading is given by $$2dA_{\mu_2}([{\bf y}, i,j])= \begin{cases} 2(1-m) & \text{ if }i=0\\ -2(m+1)& \text { if } i=-1 \\ -2(m+3) & \text{ if } i<-1 \end{cases}$$
where $d$ is the order of $[\mu_2]$ in $H_1(S^3_n(K))$, which is $n$ when $n$ is odd and $\frac{n}{2}$ when $n$ is even.
\end{lemma}

\begin{proof}
Let $(\Sigma,\b\alpha,\b\beta_0\cup\mu,w,z)$ be a Heegaard diagram for $(S^3,K)$ adapted to a Seifert surface $F$ for $K$ so that every intersection point ${\bf y}$ in the diagram is of the form $({\bf y}_0,p)$ where $p=\alpha_g\cap \mu$ and $\mathcal{P}_0$ is a relative periodic domain representing $F$ with $\partial\mathcal{P}_0=\lambda_0+\alpha_g$. 
 By replacing $\mu$ with $\lambda_n=\lambda_0+n\mu$, we have a related Heegaard diagram $(\Sigma,\b\alpha,\b\beta_0\cup\lambda_n,w,z')$ for $(S^3_n(K),\mu_2)$.

In the coordinates for the diagram for $S^3_n(K)$, we have that $\mu_2=2\mu-\alpha_g$ and $\lambda_0=\lambda_4-4\mu$. Starting with $\partial\mathcal{P}_0$ and substituting gives us a relative periodic domain $\mathcal{P}_{\mu_2^r}$ with $$\partial \mathcal{P}_{\mu_2^r}=\lambda_4-2\mu_2-\alpha_g.$$
We have chosen our notation to reflect the fact that $\mu_2$ appears in the equation for $\partial\mathcal{P}_{\mu_2^r}$ with a minus sign, which means that $\partial\mathcal{P}_{\mu_2^r}$ is naturally a periodic domain representing a rational Seifert surface for the reverse of $\mu_2$. The relative periodic domain we are interested in is $\mathcal{P}_{\mu_2}$ with boundary $\mu_2$. This domain is given by the same 2-chain as $\mathcal{P}_{\mu_2^r}$ except that the signs of all of the multiplicities are reversed. 

Now, we enumerate intersection points $\lambda_n\cap\alpha_g$ to the left of $p$ starting with $x_0$ and increasing to the left. Enumerate points to the right of $p$ starting with $x_{-1}$ and decreasing to the right as in Figure \ref{WindingRegion4}. Given the above, we can compute the Alexander gradings of points $({\bf y},x_k)$ in terms of the Alexander gradings of points ${\bf y}$,
\begin{align*}
 2dA({\bf y},x_k)&=\hat\chi(\mathcal{P}_{\mu_2})+2n_{({\bf y}_0,x_k)}(\mathcal{P}_{\mu_2})-n_{w}(\mathcal{P}_{\mu_2})-n_z(\mathcal{P}_{\mu_2})\\
 &=\hat\chi(-\mathcal{P}_{\mu_2^r})+2n_{{\bf y}_0}(-\mathcal{P}_{\mu_2^r})+2n_{x_k}(\mathcal{P}_{\mu_2})-n_{w}(\mathcal{P}_{\mu_2})-n_z(\mathcal{P}_{\mu_2})\\
   &=-\hat\chi(\mathcal{P}_{0})-2n_{{\bf y}_0}(\mathcal{P}_{0})+2n_{x_k}(\mathcal{P}_{\mu_2})-n_{w}(\mathcal{P}_{\mu_2})-n_z(\mathcal{P}_{\mu_2})\\
   &=-2A({\bf y})+2n_{x_k}(\mathcal{P}_{\mu_2})-n_{w}(\mathcal{P}_{\mu_2})-n_z(\mathcal{P}_{\mu_2}).
 \end{align*}
For points $x_k$ to the left of $p$ we have:
\begin{itemize}
\item $2n_{x_k}=2(\ell-k-1)$
\item $n_w=\ell-1$
\item $n_z=\ell-3$
\end{itemize}
where $x_\ell$ is the intersection point furthest to the left of $p$. 

For points $x_{-k}$ to the right of $p$ with $i>1$ we have:
\begin{itemize}
\item $2n_{x_{-k}}=2(-r+k-2)$
\item $n_w=-r+2$
\item $n_z=-r$
\end{itemize}
where $x_r$ is the intersection point furthest to the right of $p$. In addition for $x_{-1}$ we have $2n_{x_{-1}}(\mathcal{P})-n_w(\mathcal{P})-n_z(\mathcal{P})=0$. 

It is worth noting that the Alexander grading of each generator is independent of $\ell$ and $r$. The Alexander grading only depends on how far to the left or right of $p$ the final coordinate lies.  

Thus we have, $$2dA({\bf y}_0,x_k)=-2A({\bf y})+ \begin{cases} 2-2k & k\geq 0 \\ 0  & k=-1\\ -6-2k &i<-k. \end{cases}$$

Finally, if $[{\bf y}, i,j]\in\mathcal{C}\{\max(i,j-m)=0\}$, and $n$ is large, we have that $[{\bf y}, i,j]=\Phi_m({\bf y}_0,x_k)$ for some $k$ where $$\Phi_m: \widehat{CF}(S^3_n(K),\mathfrak{s}_m)\to \mathcal{C}\{\max(i,j-m)=0\}$$ is the chain map given by, $$\Phi_{m}[\mathbf{x}]=\sum_{y\in\intpt}\sum_{\substack{\psi\in\pi_2(\mathbf{x},\b{\theta},\mathbf{y}) \\ n_z(\psi)-n_w(\psi)=\mathcal{F}(\mathbf{y})-m \\ \mu(\psi)=0}} \#\mathcal{M}(\psi)[\mathbf{y},-n_w(\psi),m-n_z(\psi)].$$

 If $i=0$, then $j=m-n_z(\psi)=m-k$. Thus, $[{\bf y}, 0,m-k]=\Phi_m({\bf y}_0,x_k)$. So $$2dA_{\mu_2}([{\bf y}, 0,m-k])=2(k-m)+ 2-2k=2-2m.$$ On the other hand, if $i<0$ then $i=-n_w(\phi)=k$. Thus, $[{\bf y}, i,m]=\Phi_m({\bf y}_0,x_k)$. So $$2dA_{\mu_2}([{\bf y}, i,m])=2(i-m)+\begin{cases} 0 & i=-1\\ -6-2i & i<-1.\end{cases}$$ 
\end{proof}

Given $CKF^\infty(S^3,K)$ for any knot $K$ in the 3-sphere, Proposition \ref{refinement} and Lemma \ref{Alexgradings} allow us to calculate the $\tau$-invariants of $\mu_2$ in $S^3_n(K)$ where $n\geq 2g(K)+1$. 

We provide an example calculation. Let $K$ be the $5_2$-knot in the 3-sphere. Since $K$ is 2-bridge, its $CFK^\infty$-complex can be calculated from a genus one Heegaard diagram. The resulting complex is pictured in Figure \ref{CFK5_2}.

\begin{figure}
\includegraphics[scale=2]{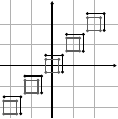}
\caption{$CFK^\infty(S^3,K)$ for the $5_2$-knot. Elements generating the ``tower" in $HF^+(S^3)$ are indicated in black.}
\label{CFK5_2}
\end{figure}

Applying Proposition \ref{refinement}, we calculate the relative filtration induced by $\mu_2$ in each $\s$-structure as shown in Figure \ref{filtrations}. Note that $g(K)=1$, so our choice of $n=4$ is sufficient. 

\begin{figure}[t]
\includegraphics[scale=1.5]{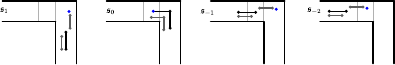}
\caption{The complexes $\widehat{CF}(S^3_4(K),\mathfrak{s}_i)$ with the filtration by $\mu_2$ indicated. Generators in gray do \emph{not} map to $\mathcal{T}^+\subset HF^+(S^3_4(K), \mathfrak{s}_i)$ and the generator that carries the $\tau$-invariant is indicated in blue. 
}
\label{filtrations}
\end{figure}

We observe that $S^3_4(K)$ is not an $L$-space, since $\widehat{HF}(S^3_4(K),\mathfrak{s}_0)=\F_2\oplus\F_2$. Applying Lemma \ref{Alexgradings} to calculate the Alexander gradings of the indicated generators we conclude that,
\begin{align*}
\tau_{\mathfrak{s}_1}(S^3_4(K),\mu_2)&=0; &\tau_{\mathfrak{s}_0}(S^3_4(K),\mu_2)&=-\frac{1}{2};\\
\tau_{\mathfrak{s}_{-1}}(S^3_4(K),\mu_2)&=1; &\tau_{\mathfrak{s}_{-2}}(S^3_4(K),\mu_2)&=\frac{1}{2}.
\end{align*}  

Finally, recall that if a $\s$-structure extends over a rational ball then its associated $d$-invariant vanishes.  Therefore, calculating the $d$-invariants of $S^3_4(K)$ will indicate which $\tau$-invariants are relevant for showing that $\mu_2$ is not slice. The $d$-invariants are given by the homological grading of the same generators indicated in Figure \ref{filtrations} that carry the $\tau$-invariants in each $\s$-structure. The necessary grading shift formulas are calculated in \cite[Theorem 4.4]{OS-knotinvariants}. See, for instance, \cite[Theorem 5.3]{Hedden-Livingston-Ruberman} for the phrasing in terms of positive surgery. This calculation shows that $d_{\mathfrak{s}_{-1}}$ and $d_{\mathfrak{s}_1}$ are zero while the other two $d$-invariants are non-zero. Therefore, these are the $\s$-structures that will extend over a rational ball. Thus, since $\tau_{\mathfrak{s}_{-1}}(S^3_4(K),\mu_2)=1$, we know that $\mu_2$ cannot be slice in any rational ball.

We conclude by assuring the reader that our finding is not vacuous. Figure \ref{5_2ball} shows that $S^3_4(K)$ does in fact bound a rational ball. 

\begin{figure}
\includegraphics[scale=.25]{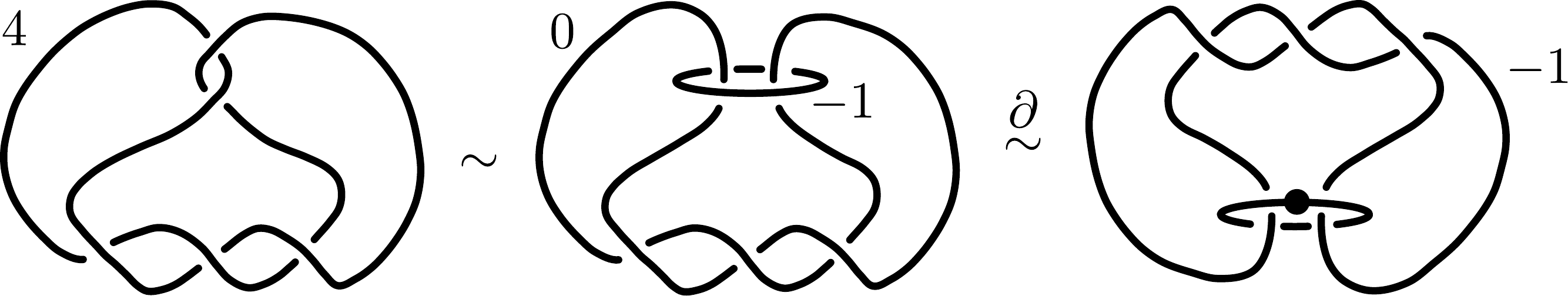}
\caption{A Kirby diagram showing that 4-surgery on the $5_2$-knot bounds a rational ball.}
\label{5_2ball}
\end{figure}

\bibliographystyle{alpha}	
\bibliography{MasterBibliography}

\begin{thebibliography}{HKL16}

\bibitem[Cel18]{Celoria}
Daniele Celoria.
\newblock On concordances in 3-manifolds.
\newblock {\em J. Topol.}, 11(1):180--200, 2018.

\bibitem[CG13]{Calegari-Gordon}
Danny Calegari and Cameron Gordon.
\newblock Knots with small rational genus.
\newblock {\em Comment. Math. Helv.}, 88(1):85--130, 2013.

\bibitem[Cha07]{Cha}
Jae~Choon Cha.
\newblock The structure of the rational concordance group of knots.
\newblock {\em Mem. Amer. Math. Soc.}, 189(885):x+95, 2007.

\bibitem[Hed07]{Hedden-Whitehead}
Matthew Hedden.
\newblock Knot {F}loer homology of {W}hitehead doubles.
\newblock {\em Geom. Topol.}, 11:2277--2338, 2007.

\bibitem[HKL16]{Hedden-Kim-Livingston}
Matthew Hedden, Se-Goo Kim, and Charles Livingston.
\newblock Topologically slice knots of smooth concordance order two.
\newblock {\em J. Differential Geom.}, 102(3):353--393, 2016.

\bibitem[HL19]{Hedden-Levine-surgery}
Matthew {Hedden} and Adam~Simon {Levine}.
\newblock {A surgery formula for knot Floer homology}.
\newblock {\em arXiv e-prints}, page arXiv:1901.02488, January 2019.

\bibitem[HLR12]{Hedden-Livingston-Ruberman}
Matthew Hedden, Charles Livingston, and Daniel Ruberman.
\newblock Topologically slice knots with nontrivial {A}lexander polynomial.
\newblock {\em Adv. Math.}, 231(2):913--939, 2012.

\bibitem[HP13]{Hedden-Plam}
Matthew Hedden and Olga Plamenevskaya.
\newblock Dehn surgery, rational open books and knot {F}loer homology.
\newblock {\em Algebr. Geom. Topol.}, 13(3):1815--1856, 2013.

\bibitem[MT15]{Mark-Tosun}
T.~E. Mark and B.~Tosun.
\newblock Naturality of {H}eegaard {F}loer invariants under positive rational
  contact surgery.
\newblock {\em ArXiv e-prints}, November 2015.
\newblock {\tt arXiv:1509.01511}.

\bibitem[Ni09]{Ni-LinkFloer}
Yi~Ni.
\newblock Link {F}loer homology detects the {T}hurston norm.
\newblock {\em Geom. Topol.}, 13(5):2991--3019, 2009.

\bibitem[NV16]{Ni-Vafaee}
Yi~Ni and Faramarz Vafaee.
\newblock Null surgery on knots in {L}-spaces.
\newblock {\em ArXiv e-prints}, August 2016.
\newblock {\tt arXiv:1608.07050}.

\bibitem[NW14]{Ni-Wu}
Yi~Ni and Zhongtao Wu.
\newblock Heegaard {F}loer correction terms and rational genus bounds.
\newblock {\em Adv. Math.}, 267:360--380, 2014.

\bibitem[OS03]{OS-4ball}
Peter Ozsv\'ath and Zolt\'an Szab\'o.
\newblock Knot {F}loer homology and the four-ball genus.
\newblock {\em Geom. Topol.}, 7:615--639, 2003.

\bibitem[OS04a]{OS-knotinvariants}
Peter Ozsv\'ath and Zolt\'an Szab\'o.
\newblock Holomorphic disks and knot invariants.
\newblock {\em Adv. Math.}, 186(1):58--116, 2004.

\bibitem[OS04b]{OS-properties}
Peter Ozsv\'ath and Zolt\'an Szab\'o.
\newblock Holomorphic disks and three-manifold invariants: properties and
  applications.
\newblock {\em Ann. of Math. (2)}, 159(3):1159--1245, 2004.

\bibitem[OS04c]{OS-threemanifolds}
Peter Ozsv\'ath and Zolt\'an Szab\'o.
\newblock Holomorphic disks and topological invariants for closed
  three-manifolds.
\newblock {\em Ann. of Math. (2)}, 159(3):1027--1158, 2004.

\bibitem[OS08]{OS-multivariable}
Peter Ozsv\'ath and Zolt\'an Szab\'o.
\newblock Holomorphic disks, link invariants and the multi-variable {A}lexander
  polynomial.
\newblock {\em Algebr. Geom. Topol.}, 8(2):615--692, 2008.

\bibitem[OS11]{OS-rationalsurgeries}
Peter~S. Ozsv\'ath and Zolt\'an Szab\'o.
\newblock Knot {F}loer homology and rational surgeries.
\newblock {\em Algebr. Geom. Topol.}, 11(1):1--68, 2011.

\bibitem[Ras03]{Rasmussen-knots}
Jacob~Andrew Rasmussen.
\newblock {\em Floer homology and knot complements}.
\newblock ProQuest LLC, Ann Arbor, MI, 2003.
\newblock Thesis (Ph.D.)--Harvard University.

\bibitem[Tru19]{Truong}
Linh Truong.
\newblock A refinement of the ozsv\'ath-szab'o large integer surgery formula
  and knot concordance.
\newblock {\em arXiv e-prints}, page arXiv:1904.00288, Mar 2019.

\bibitem[Tur97]{Turaev-Spinc}
Vladimir Turaev.
\newblock Torsion invariants of {${\rm Spin}^c$}-structures on {$3$}-manifolds.
\newblock {\em Math. Res. Lett.}, 4(5):679--695, 1997.

\end{thebibliography}

\end{document}